\documentclass[10pt]{amsart}
\usepackage{upref}
\usepackage{color}
\usepackage[colorlinks=true, linkcolor=blue, citecolor=red] {hyperref}
\newcommand*{\mailto}[1]{\href{mailto:#1}{\nolinkurl{#1}}}

\usepackage{graphicx}
\graphicspath{{figs/}}

\newtheorem{theorem}{Theorem}[section]

\newtheorem{lemma}[theorem]{Lemma}
\newtheorem{definition}[theorem]{Definition}

\theoremstyle{definition}
\newtheorem{remark}[theorem]{Remark}

\numberwithin{equation}{section}     
\allowdisplaybreaks

\renewcommand{\i}{\ifmmode\mathit{\mathchar"7010 }\else\char"10 \fi}
\renewcommand{\j}{\ifmmode\mathit{\mathchar"7011 }\else\char"11 \fi}
\newcommand{\R}{\mathbb{R}}

\newcommand{\abs}[1]{\left|#1\right|}
\newcommand{\norm}[1]{\left\|#1\right\|}

\newcommand{\pt}{\partial_t}
\newcommand{\px}{\partial_x}

\newcommand{\ps}{\partial_s}
 \newcommand{\ptt}{\partial_{tt}^2}

\newcommand{\test}{\varphi}
\newcommand{\eps}{\varepsilon}

\newcommand{\vm}{v^\mu}
\newcommand{\abseps}[1]{\abs{#1}_\eps}

\newcommand{\um}{u^\mu}

\newcommand{\Dx}{{\Delta x}}
\newcommand{\Dt}{{\Delta t}}
\newcommand{\unpj}{u^{n+1}_j}
\newcommand{\unj}{u^n_j}
\newcommand{\vnj}{v^n_{j+1/2}}
\newcommand{\vnpj}{v^{n+1}_{j+1/2}}
\newcommand{\vnmj}{v^{n+1}_{j-1/2}}
\newcommand{\kjm}{k_{j-1/2}}
\newcommand{\kjp}{k_{j+1/2}}
\newcommand{\Dp}{\Delta_+}
\newcommand{\Dm}{\Delta_-}
\newcommand{\ddm}{D_-}
\newcommand{\ddp}{D_+}
\newcommand{\dpt}{D^t_+}

\makeatletter
\@namedef{subjclassname@2020}{%
  \textup{2020} Mathematics Subject Classification}
\makeatother

\DeclareMathOperator*{\sign}{sign}
\newcommand{\sgn}[1]{\sign\left(#1\right)}

\newcommand{\dv}[1]{\mathrm{div}\left(#1\right)}
\newcommand{\dvx}[1]{\mathrm{div}_x\left(#1\right)}
\newcommand{\dvy}[1]{\mathrm{div}_y\left(#1\right)}
\newcommand{\dvxy}[1]{\nabla^2_{xy}\left(#1\right)}
\newcommand{\signe}{\sign\nolimits_\eps}

\newcommand{\seq}[1]{\left\{#1\right\}}
\newcommand{\order}[1]{\mathcal{O}\left(#1\right)}

\DeclareMathOperator*{\supp}{supp}
\newcommand{\dott}{\, \cdot\,}

\newenvironment{Assumptions}
{%

\begin{enumerate}}%
{\end{enumerate}}

{%

\begin{enumerate}}%
{\end{enumerate}}

\title[Singular Diffusion]{Singular Diffusion \\ with Neumann boundary
  conditions} 

\author[G. M.Coclite]{Giuseppe Maria Coclite}
\address[G. M. Coclite]{Department of Mechanics, Mathematics and Management\\
Polytechnic University of Bari\\
Via E. Orabona 4\\
70125 Bari \\
Italy}
\email{\mailto{giuseppemaria.coclite@poliba.it}}
\urladdr{\url{https://sites.google.com/site/coclitegm/}}

\author[H. Holden]{Helge Holden}
\address[H. Holden]{Department of Mathematical Sciences\\
  NTNU Norwegian University of Science and Technology\\
  NO-7491 Trondheim\\ Norway}
\email{\mailto{helge.holden@ntnu.no}}
\urladdr{\url{https://www.ntnu.edu/employees/holden}}

\author[N. H. Risebro]{Nils Henrik Risebro}
\address[N. H. Risebro]{Department of Mathematics\\
   University of Oslo\\
  NO-0316 Oslo\\ Norway}
\email{\mailto{nilshr@math.uio.no}}
\urladdr{\url{https://www.mn.uio.no/math/english/people/aca/nilshr/index.html}}

\date{\today}
\thanks{This research was jointly and partially supported 
by the Research Council of Norway Toppforsk 
project {\em Waves and Nonlinear Phenomena} (250070), it has also received funding from the European Union's
Horizon 2020 research and innovation programme under the Marie Sk\l{}odowska-Curie grant
agreement No 642768.\\
\indent GMC is member of the Gruppo Nazionale per l'Analisi Matematica, la Probabilit\`a e le loro Applicazioni (GNAMPA) 
of the Istituto Nazionale di Alta Matematica (INdAM)}

\begin{document}

\begin{abstract}
  In this paper we develop an existence theory  
  for the nonlinear initial-boundary value problem with singular
  diffusion $\pt u = \dv{k(x)\nabla G(u)}$, $u|_{t=0}=u_0$ with
  Neumann boundary conditions $k(x)\nabla G(u)\cdot \nu = 0$.  Here
  $x\in B\subset \R^d$, a bounded open set with locally Lipchitz
  boundary, and with $\nu$ as the unit outer normal. The function $G$ is
  Lipschitz continuous and nondecreasing, while $k(x)$ is diagonal
  matrix. 
  We show that  any two weak entropy solutions $u$
  and $v$ satisfy
  $\norm{u(t)-v(t)}_{L^1(B)}\le
  \norm{u|_{t=0}-v|_{t=0}}_{L^1(B)}e^{Ct}$, for almost every $t\ge 0$,
  and a constant $C=C(k,G,B)$.
 If we restrict to the case when  the entries $k_i$ of $k$ depend only on the corresponding component, $k_i=k_i(x_i)$,
 we show that there exists an entropy solution, thus establishing in this case that the problem is well-posed  in the sense of Hadamard.
\end{abstract}

\maketitle

\section{Introduction}

We here study the nonlinear boundary value problem with singular
diffusion
\begin{equation*}
  \pt u = \dv{k(x)\nabla G(u)}, \quad u|_{t=0}=u_0
\end{equation*}
with Neumann boundary conditions
\begin{equation*}
  k(x)\nabla G(u)\cdot \nu = 0.
\end{equation*}
Here $x\in B\subset \R^d$, a bounded open set with locally Lipchitz
boundary, and with $\nu$ as the unit outer normal at the boundary of $B$.
The function $G$ is Lipschitz continuous and nondecreasing, while
$k(x)=\text{diag}(k_1(x_1),\dots,k_d(x_d))$ is a smooth, diagonal matrix
with positive entries $k_i>0$.  The fact that $G'$ may vanish, allows
for singular diffusion, and the factor $k$ permits for spatially
dependent diffusion.  Due to the singular diffusion, this equation has
weak solutions.

Our goal is to prove well-posedness of this equation in appropriate
spaces. Working with weak solutions, we need to impose an appropriate
entropy condition, and we here demand, see Definition
\ref{def:entrsolk_new}, that
\begin{align*}
  \int_{\Omega_T} \big[\eta(u)\pt \test - \eta'(u) (k(x) \nabla G(u)) \nabla\test\big] \,dxdt  
  &-\Big|_{t=0}^{t=T}\int_B \eta(u)\test\,dx   \\
  &\quad \ge \int_{\Omega_T} \eta''(u) \left|h(x)\nabla g(u)\right|^2 \test\,dxdt
\end{align*}
holds for every smooth convex entropy $\eta$ and for all nonnegative
test functions $\phi$.  Here $\Omega_T=[0,T]\times B$.  One of our main
theorem reads, see Theorem \ref{thm:maink}, as follows:

\textit{ 
For two weak entropy solutions $u$ and $v$ with initial data
  $u_0$ and $v_0$, respectively, we have
  \begin{equation}\label{eq:introStab}
    \norm{u(t)-v(t)}_{L^1(B)}\le \norm{u_0-v_0}_{L^1(B)}e^{Ct},
  \end{equation}
  for almost every $t\ge 0$. The constant $C$ depends on $k$, $G$, and $B$. 
  Here we only need to assume that $k$ is a diagonal matrix.
  }

As is common, existence and stability are proven by two independent
arguments. Assuming existence of two weak entropy solutions $u$ and
$v$, we proceed using the ingenious doubling of variables due to
Kru\v{z}kov, see \cite{MR3443431}. In this approach, one starts with
entropy conditions for $u$ (in variables $(t,x)$) and $v$ (in
variables $(s,y)$), and considers a test function $\test$ that depends
on all four variables $(t,x,s,y)$.  Integrating over all variables,
and adding the entropy expressions for $u$ and $v$, one finds the
inequality \eqref{eq:en0}.  The next step is the choice of test
function. Here we follow in the steps of Kru\v{z}kov and choose an
approximate Dirac delta function in the time variable $t-s$ and space
variables $x-y$ in addition to a smooth cut-off function near the
boundary.  See equation \eqref{eq:testK}.  Next comes the delicate limits. We
start by taking the $s\to t$ limit. The fact that we deal with a bounded domain 
goes beyond the standard Kru\v{z}kov theory. We find that we need to
remove the  spatial singularity and the cut-off at the boundary 
simultaneously, and, importantly, at a fixed ratio where the spatial 
singularity has to vanish at a faster rate than the boundary cut-off, see
\eqref{eq:Cest}. This proves equation \eqref{eq:introStab}.

As for existence of weak entropy solutions, our starting point is the
regularized equation where we add non-degenerate diffusion. Thus we
consider for $\mu$ positive
\begin{equation*}
  \pt \um = \dv{k(x)\nabla G(\um)}+\mu\Delta \um,\quad  \um|_{t=0}=u_{0}
\end{equation*}
with boundary condition
$\big(k(x)\nabla G(\um(t,x))+\mu\nabla \um(t,x)\big)\cdot\nu=0$.
Existence of solutions for this problem follows from \cite[Theorem
1.7.8]{MPS}.

In our case we have only been able to prove existence of a solution
under the restriction that $k_i=k_i(x_i)$. Our proof is based on
showing compactness of the sequence $\seq{\um}_{\mu> 0}$, we were
not able to deduce this without the simplifying assumption that $k_i$
only depends on $x_i$.  Our main existence result, Theorem
\ref{th:existence}, reads as follows:

\textit{ Assume that $k_i=k_i(x_i)>0$. Then there exists a weak entropy
  solution $u$ the initial-boundary value problem. In particular, we
  have, as $\mu\to 0$, that
  \begin{equation*}
    \um\to u\quad\text{a.e. and in $L^p(\Omega_T)$ for every $T>0$ and $1\le p<\infty$.}
  \end{equation*}
Thus the problem is well-posed in the sense of Hadamard.}

The proof starts with the entropy formulation for smooths solutions,
which yields (cf.~\eqref{eq:smooth_entropy})
\begin{equation*}
  \frac{d}{dt}\int_B \eta(\um)\,dx + \int_B  \eta^{\prime\prime}(\um)G'(\um)|h(x)\nabla\um|^2  \,dx 
  + \mu\int_B  \eta^{\prime\prime}(\um)|\nabla\um|^2  \,dx   = 0
\end{equation*}
for convex entropies $\eta$.  By choosing different functions for the
entropy, we can show a wide range of properties of the approximate
solution $\um$, see Lemma \ref{lm:est1}.  In particular, we show that
$\norm{\nabla G(\um)}_{L^2(\Omega_T)}$ is bounded independently
of $\mu$.  Furthermore, we verify that
\begin{equation*}
  \norm{\partial_{x_j} \um(t)}_{L^1(B_\sigma)}\le \norm{\partial_{x_j} u_0}_{L^1(B)}e^{ C_\sigma t},
\end{equation*}
where $B_\sigma$ is the subset of $B$ with distance at least $\sigma$
from $\partial B$, holds. From these estimates, we can prove, as
$\mu\to0$, that $\um\to u$ in $L^p(\Omega_T)$ for every $T>0$ and
$1\le p<\infty$.

The problem of analyzing parabolic equations with singular diffusion
has of course been studied by several researchers, and the literature
is too comprehensive to be discussed in detail here.  Our paper relies
on the seminal paper by Carrillo \cite{MR1709116} where the
Kru\v{z}kov doubling of variables, see, e.g., \cite{MR3443431}, is
applied to study multi-dimensional degenerate parabolic problems, and
with an entropy condition due to Vol'pert--Hudjaev
\cite{MR0264232}. Carrillo studied the case of Dirichlet boundary
conditions.  An early result can be found in \cite{MR539218}.  Further
generalizations of the results by Carrillo can be found in \cite{MR2187640,
  MR2149515} by Chen and Karlsen, extending the analysis to the whole
space and allowing for spatial and temporal dependence the various
terms, as well as nonlinear transport. We have relied on the work by Karlsen and 
Ohlberger  \cite{MR1941794}, in particular their result regarding the weak chain rule, 
see  equation \eqref{eq:weakchain}.  Different boundary conditions
in one dimension are studied in \cite{MR1769093}. Karlsen and Risebro
\cite{MR1974417} studied uniqueness and stability of nonlinear
degenerate parabolic equations with rough coefficients. See also
\cite{Chen2001}.

In Section \ref{sec:num} we study a convergent difference scheme for
this equation in one dimension with $B=(0,1)$. Let $u^n_j$ be an approximation to
$u(t_n,x_j)$ with $t_n=n\Dt$ and $x_j=(j+1/2)\Dx$ for small positive numbers
$\Dt,\Dx$. The implicit first-order difference scheme is given by
\begin{equation*}
  \unpj - \mu \Dp\left(k_{j-1/2}\Dm G\left(\unpj\right)\right) = \unj, \ j=0,\ldots,N,
\end{equation*}
with boundary conditions that $\Dm u^{n+1}_0=\Dp u^n_N=0$. Here $\Dp$
($\Dm$) is the forward (backward) spatial difference,
$k_{j-1/2}=k(x_{j-1/2})$, and $\mu=\Dt/\Dx^2$ (assumed to be bounded
from below). We define the function $u_\Dt(t,x)$ on
$[0,\infty)\times B$ by making it equal to $\unj$ on the rectangle
$[t_n,t_{n+1})\times [x_{j-1/2},x_{j+1/2})$.  See \cite{MR3246807} for
related results on numerical methods.

Our interest in this equation stems from the modeling of multilane
dense traffic. Each lane is modeled by the traditional
Lighthill--Whitham--Richards model, which gives a scalar hyperbolic
conservation law where the unknown function describes the density of
vehicles. Our aim was to model unidirectional multilane traffic. We
discovered, see \cite{MR3721873, MR3845580, MR4001759}, that the model
we studied, allowed for an infinite number of lanes limit, resulting
in an equation resembling the one of this paper.


\section{Space dependent singular diffusion}\label{sec:diffusion}
We are interested in the boundary value problem
\begin{equation}
  \label{eq:boundval}
  \begin{cases}  
    \pt u = \dv{k(x)\nabla G(u)},& (t,x)\in (0,\infty)\times B,\\
    (k(x)\nabla G(u))\cdot \nu = 0,& (t,x)\in (0,\infty)\times\partial B,\\
    u(0,x)=u_0(x),& x\in B,
  \end{cases}
\end{equation}
where we shall assume that
\begin{Assumptions}
\item \label{ass:B} $B\subset \R^d$ is a bounded open set with locally
  Lipchitz boundary and $\nu$ the unit outer normal to its boundary;
\item \label{ass:k} $k\in C^2(\overline{B};\R^{d\times d})$ is a
  diagonal matrix
  \begin{equation*}
    k(x)=\left(
      \begin{matrix}
        k_1(x)&\dots&0\\
        \vdots&\ddots&\vdots\\
        0&\dots&k_d(x)
      \end{matrix}
    \right),\qquad x\in \overline{B},
  \end{equation*}
  with $k_i\in C^2(\overline{B})$ and
  $k_i(x)>0$; 
\item \label{ass:G} $G$ is Lipschitz continuous and nondecreasing;
\item \label{ass:init} $u_0\in W^{2,1}(B)\cap L^\infty(B)$.
\end{Assumptions}

We use the notation $\Omega=[0,\infty)\times B$ and
$\Omega_T=[0,T]\times B$.  It is also useful to define
\begin{equation*}
  h(x)=\left(
    \begin{matrix}
      h_1(x)&\dots&0\\
      \vdots&\ddots&\vdots\\
      0&\dots& h_d(x)
    \end{matrix}
  \right),\quad h_i(x)=\sqrt{k_i(x)},\quad  g(u)=\int^u \sqrt{G'(\xi)}\,d\xi.
\end{equation*}
\begin{definition}
  \label{def:weaksolk}
  A function $u\in C([0,\infty);L^1(B))$ is a weak solution of
  \eqref{eq:boundval} if
  \begin{align}
    \label{eq:visck}
    &u\in L^\infty(\Omega),\\
    \label{eq:visc1k}
    &\nabla G(u)\in L^2(\Omega_T;\R^d),\\
    \label{eq:visc2k}
    &(k(x)\nabla G(u))\cdot \nu = 0\> \text{in the sense of traces on $\partial B$ for a.e. $t$},
  \end{align}
  and for every test function $\test\in C^\infty_0(\Omega)$
  \begin{equation}
    \label{eq:defweak}
    \int_\Omega \big[u\pt\test- (k(x) \nabla G(u))\cdot\nabla\test\big] \,dtdx+\int_B u_0(x)\test(0,x) \,dx=0.
  \end{equation}
\end{definition}
\begin{definition}
  \label{def:entrsolk_new}
  A function $u\in C([0,\infty);L^1(B))$ is an entropy solution of
  \eqref{eq:boundval} if it is a weak solution of \eqref{eq:boundval}
  in the sense of Definition \ref{def:weaksolk} and for every convex
  entropy $\eta\in C^2(\R)$ and for all nonnegative test functions
  $\test\in C^\infty_0(\Omega)$
  \begin{align}
    \int_{\Omega_T} \big[\eta(u)\pt \test - &\eta'(u) (k(x) \nabla G(u)) \nabla\test\big] \,dxdt  \notag\\
                                            &\quad -\int_B \eta(u(T,x))\test(T,x)\,dx   + \int_B \eta(u_0(x))\test(0,x)\,dx      \label{eq:entrknew}  \\
                                            &\qquad\qquad\qquad\qquad\qquad  \ge \int_{\Omega_T} \eta''(u) \left|h(x)\nabla g(u)\right|^2 \test\,dxdt.\notag
  \end{align}
\end{definition}
Recall that the following weak chain rule holds, see \cite{MR1941794}.
Let $A$ be a continuous function on $[0,\infty)$ with $A(0)=0$ and $b$
a continuous function. If $u$ is an entropy solution in the sense of
Definition~\ref{def:entrsolk_new} then
\begin{equation}
  \label{eq:weakchain}
  b(u)\nabla\Bigl(\int^u A\left(\sqrt{G'(\xi)}\right)\,d\xi\Bigr)  =
  \nabla \Bigl(\int^u b(\xi)A\left(\sqrt{G'(\xi)}\right)\,d\xi\Bigr),
\end{equation}
weakly in $\Omega$. Then the following theorem holds.
\begin{theorem}[{\bf Uniqueness and Stability}]
  \label{thm:maink}
  Assume that \ref{ass:B}, \ref{ass:k}, \ref{ass:G}, and
  \ref{ass:init} hold. 
  Let $u$ and $v$ be
  two weak entropy solutions of the  initial-boundary value problem \eqref{eq:boundval}.  Then
  \begin{equation}
    \label{eq:stabk}
    \norm{u(t,\dott)-v(t,\dott)}_{L^1(B)}\le \norm{u(0,\dott)-v(0,\dott)}_{L^1(B)}e^{Ct},
  \end{equation}
  for almost every $t\ge 0$. Here $C$ is a finite positive constant
  depending on $k$, $G$, and $B$.
\end{theorem}
\begin{proof}
  We define the signum function as
  \begin{equation*}
    \sgn{\xi}=
    \begin{cases}
      1&\xi>0,\\ 0 &\xi=0,\\ -1 &\xi<0,
    \end{cases}
  \end{equation*}
  and its regularized version
  \begin{equation*}
    \signe\left(\xi\right)=
    \begin{cases}
      1&\xi>\eps,\\ \sin\bigl(\frac{\pi\xi}{2\eps}\bigr) &\abs{\xi}\le
      \eps,\\ -1 &\xi<-\eps.
    \end{cases}
  \end{equation*}
  Now set
  \begin{equation*}
    \abseps{\sigma}=\int_0^\sigma \signe(\xi)\,d\xi, \ \ \eta_\eps(u,v)=\abseps{u-v}.
  \end{equation*}
  Let now $u=u(t,x)$ and $v=v(s,y)$ be two entropy solutions with
  initial data $u_0$ and $v_0$, respectively. The maps
  $u\mapsto \eta_\eps(u,v)$ and $v\mapsto \eta_\eps(u,v)$ are
  admissible entropies. Let $\test=\test(t,x,s,y)$ be an admissible
  test function both in $(t,x)$ and in $(s,y)$.  By adding the entropy
  condition for $u$ and $v$ we get
  \begin{align}
    \int_{\Omega_T^2}\big[ &\abseps{u-v}\left(\pt + \ps\right)\test -
                             \signe(u-v) (k(x)\nabla_x G(u) )\cdot \nabla_x\test \notag\\
                           &\qquad\qquad\qquad\qquad\qquad\qquad+ \signe(u-v)(k(y)\nabla_y G(v))\cdot\nabla_y\test\big] \,dX \notag\\
                           &\quad - \int_{B\times\Omega_T}\!\!\!\!
                             \abseps{u-v}\test\,dxdyds\Bigm|_{t=0}^{t=T} - \int_{\Omega_T\times
                             B} \!\!\!\!
                             \abseps{u-v}\test\,dxdtdy\Bigm|_{s=0}^{s=T} \notag\\
    \ge& \int_{\Omega_T^2} \signe'(u-v)\left[\left|h(x) \nabla_x g(u)\right|^2 + \left|h(y)\nabla_y g(v)\right|^2\right]\test\,dX,
         \label{eq:en0}
  \end{align}
  where
  \begin{equation*}
    dX=dtdxdsdy,
  \end{equation*}
  with $dx=dx_1\dots dx_d$ and similarly for $dy=dy_1\dots dy_d$.
  Using the basic inequality $a^2+b^2 \ge \pm 2ab$, this can be
  rewritten
  \begin{equation}\label{eq:en1}\begin{split}
      \tau_\eps(\test) & + i_\eps(\test)-f_{\eps}(\test) - B_\eps(\test) \\
      & + \int_{\Omega_T^2}\signe(u-v)\big[(k(x)  \nabla_x G(u) )\cdot\nabla_y\test -   (k(y)\nabla_y G(v))\cdot\nabla_x \test\big] \,dX \\
      &\qquad\ge \pm 2 \int_{\Omega_T^2} \signe'(u-v)(h(x)\nabla_x
      g(u)) \cdot(h(y)\nabla_y g(v) ) \test \,dX,
    \end{split}\end{equation}
  where
  \begin{align*}
    \tau_\eps(\test)&=\int_{\Omega_T^2}\abseps{u-v}\left(\pt+\ps\right)\test\,dX,\\
    i_\eps(\test)&= \int_{B\times\Omega_T}\!\!\!\!
                   \abseps{u(0,x)-v(s,y)}\test(0,x,s,y) \,dxdsdy \\
                    &\qquad+
                      \int_{\Omega_T\times
                      B}\!\!\!\! \abseps{u(t,x)-v(0,y)} \test(t,x,0,y)\, dtdxdy,\\
    f_\eps(\test)&= \int_{B\times\Omega_T}\!\!\!\!
                   \abseps{u(T,x)-v(s,y)}\test(T,x,s,y) \,dxdsdy \\
                    &\qquad+
                      \int_{\Omega_T\times
                      B}\!\!\!\! \abseps{u(t,x)-v(T,y)} \test(t,x,T,y)\, dtdxdy,\\
    B_\eps(\test)&=\int_{\Omega_T^2} \signe(u-v)\left[ k(x)\nabla_x G(u)
                   - k(y) \nabla_y G(v)\right]\cdot\left(\nabla_x + \nabla_y\right)\test\,dX.
  \end{align*}
  By the weak chain rule \eqref{eq:weakchain}
  \begin{align*}
    \signe(u-v)\nabla_x G(u) 
    & = \signe(u-v) \nabla_x\Bigl(\int_v^u
      \gamma^2(\xi)\,d\xi
      \Bigr)\\
    &=\nabla_x\Bigl(\int_v^u \signe(\xi-v) \gamma^2(\xi)\,d\xi\Bigr),\\
    \signe(u-v)\nabla_y G(v) &= -\signe(u-v) \nabla_y\Bigl(\int_v^u
                               \gamma^2(\xi)\,d\xi \Bigr)\\
    &=-\nabla_y\Bigl(\int_v^u \signe(u-\xi)
      \gamma^2(\xi)\,d\xi\Bigr),
  \end{align*}
  where $\gamma(u)=g'(u)=\sqrt{G'(u)}$. Therefore, using the diagonal
  structure of $k$ and the fact that $\test=\test(t,x,s,y)$ vanishes
  for $x\in \partial B$ or $y\in\partial B$, we have
  \begin{align*}
    \int_{\Omega_T^2}&\signe(u-v)\big[(k(x)  \nabla_x G(u) )\cdot\nabla_y\test -   (k(y)\nabla_y G(v))\cdot\nabla_x \test\big] \,dX\\
                     &=\int_{\Omega_T^2}\big[(k(x)\nabla_y\test)\cdot(\signe(u-v)\nabla_x G(u))\\
                     &\qquad\qquad-(k(y)\nabla_x\test)\cdot (\signe(u-v)\nabla_y G(v))\big] \,dX\\
                     &=\int_{\Omega_T^2}\Bigl[(k(x)\nabla_y\test)\cdot\nabla_x\Bigl(\int_v^u \signe(\xi-v) \gamma^2(\xi)\,d\xi\Bigr)\\
                     &\qquad\qquad+(k(y)\nabla_x\test)\cdot \nabla_y\Bigl(\int_v^u \signe(u-\xi)
                       \gamma^2(\xi)\,d\xi\Bigr)\Bigr] \,dX\\
                     &=-\int_{\Omega_T^2}\Bigl[\dvxy{k(x)\test}\Bigl(\int_v^u \signe(\xi-v) \gamma^2(\xi)\,d\xi\Bigr)\\
                     &\qquad\qquad+\dvxy{k(y)\test}\Bigl(\int_v^u \signe(u-\xi)
                       \gamma^2(\xi)\,d\xi\Bigr)\Bigr] \,dX,\\
  \end{align*}
  where
  \begin{align*}
    \dvx{k(x)}&=\left(\begin{matrix}\partial_{x_1}{k_1(x)}\\\vdots\\\partial_{x_d}{k_d(x)}\end{matrix}\right),\\
    \dvxy{k(x)\test}&=
    \sum_{i=1}^d\partial_{x_i}\partial_{y_i}\big({k_i(x)\test}\big).
  \end{align*}
  Similarly,
  \begin{align*}
    \signe'(u-v)\nabla_x g(u)&=\signe'(u-v)\nabla_x\Bigl(\int_v^u
                               \gamma(\xi)\,d\xi\Bigr)\\
                             &=
                               \nabla_x\Bigl(\int_v^u \signe'(\xi-v) \gamma(\xi)\,d\xi\Bigr),  \\
    \signe'(\xi-v)\nabla_y g(v) &= \signe'(\xi-v)\nabla_y\Bigl(\int_\xi^v
                                  \gamma(\sigma)\,d\sigma\Bigr) \\ &= \nabla_y\Bigl(\int_\xi^v
                                                                     \signe'(\xi-\sigma) \gamma(\sigma)\,d\sigma\Bigr),
  \end{align*}
  so that
  \begin{equation*}
    \signe'(u-v)(\nabla_x g(u)\cdot\nabla_y g(v)) =
    \dvxy{\int_v^u\int_\xi^v \signe'(\xi-\sigma)
      \gamma(\xi)\gamma(\sigma)\,d\sigma d\xi}.
  \end{equation*}  
  Using the diagonal structure of $h$ and the fact that
  $\test=\test(t,x,s,y)$ vanishes for $x\in \partial B$ or
  $y\in\partial B$, we have
  \begin{align*}
    \int_{\Omega_T^2} &\signe'(u-v)(h(x)\nabla_x g(u)) \cdot(h(y)\nabla_y g(v) )\test \,dX\\
                      &=\int_{\Omega_T^2} \signe'(u-v)(h(x)h(y))(\nabla_x g(u) \cdot\nabla_y g(v) )\test \,dX\\
                      &=\int_{\Omega_T^2} \Bigl( \int_v^u\int_\xi^v \signe'(\xi-\sigma)
                        \gamma(\xi)\gamma(\sigma)\,d\sigma d\xi\Bigr)\dvxy{h(x)h(y)\test} \,dX.
  \end{align*}
  Then, \eqref{eq:en1} can be written
  \begin{equation} \label{eq:plusminus}
    \begin{split}
      \tau_\eps(\test)& + i_\eps(\test)-f_{\eps}(\test) - B_\eps(\test) \\
      &- \int_{\Omega_T^2} \Big[\Bigl(\int_v^u
      \signe(\xi-v)\gamma^2(\xi)\,d\xi\Bigr)
      \dvxy{k(x)\test} \\
      &+
      \Bigl(\int_v^u \signe(u-\xi)\gamma^2(\xi)\,d\xi\Bigr) \dvxy{(k(y)\test}\Big] \,dX  \\
       \ge& \pm \int_{\Omega_T^2} \Bigl( \int_v^u\int_\xi^v
      \signe'(\xi-\sigma) \gamma(\xi)\gamma(\sigma)\,d\sigma
      d\xi\Bigr) \dvxy{2h(x)h(y) \test} \,dX.
    \end{split}
  \end{equation}
  Next, we observe that
  \begin{equation*}
    \lim_{\eps\to 0} \int_\xi^v \signe'(\xi-\sigma) \gamma(\sigma)\, d\sigma = -
    \sgn{\xi-v} \gamma(\xi),
  \end{equation*}
  and consequently
  \begin{align*}
    \lim_{\eps\to 0}\int_v^u    \signe(\xi-v)\gamma^2(\xi)\,d\xi&=\Gamma(u,v),\\
    \lim_{\eps\to 0}\int_v^u \signe(u-\xi)\gamma^2(\xi)\,d\xi&=\Gamma(u,v),\\
    \lim_{\eps\to 0} \int_v^u \int_\xi^v \signe'(\xi-\sigma)    \gamma(\sigma)\gamma(\xi)\,d\sigma d\xi &=-\Gamma(u,v),
  \end{align*}
  where
  \begin{equation*}
    \Gamma(u,v):=  \int_v^u \sgn{\xi-v} \gamma^2(\xi)\,d\xi.
  \end{equation*}
  Choosing the plus sign in the above inequality \eqref{eq:plusminus}
  and using the definition of $h$, we get
  \begin{equation}\label{eq:goodestimate}\begin{split}
      \tau(\test) &+ i(\test)-f(\test) - B(\test)\\
      &\ge
      -\int_{\Omega_T^2}  \Gamma(u,v) \dvxy{\left(k(x)-2h(x)h(y)+k(y)\right)\test} \,dX\\
      &= -\int_{\Omega_T^2} \Gamma(u,v)
      \dvxy{\left(h(x)-h(y)\right)^2\test} \,dX=:- C(\test),
    \end{split} \end{equation} with
  \begin{align*}
    \tau(\test)&=\int_{\Omega_T^2} \abs{u-v}\left(\pt+\ps\right)\test
                 \,dX,\\
    i(\test)&= \int_{B\times\Omega_T}\!\!\!\!
              \abs{u_0-v}\test(0,x,s,y)\,dxdsdy \\
               &\quad+ \int_{\Omega_T\times
                 B}\!\!\!\!
                 \abs{u-v_0}\test(t,x,0,y)\,dtdxdy,\\
    f(\test)&= \int_{B\times\Omega_T}\!\!\!\!
              \abs{u(T,x)-v(s,y)}\test(T,x,s,y)\,dxdsdy \\
               &\quad+ \int_{\Omega_T\times
                 B}\!\!\!\!
                 \abs{u(t,x)-v(T,y)}\test(t,x,T,y)\,dtdxdy,\\
    B(\test)&=\int_{\Omega_T^2} \sgn{u-v}\left[ k(x)      \nabla_x G(u) - k(y)\nabla_y G(v)\right]\left(\nabla_x+\nabla_y\right)\test\,dX.
  \end{align*}
    
  The inequality \eqref{eq:goodestimate} implies the bound
  \eqref{eq:stabk} as we shall now demonstrate.  We choose a suitable
  test function $\test$. To this end let $\omega_\sigma(\xi)$ be a
  standard\footnote{We let
    $\omega_\sigma(\xi)=\frac1\sigma\omega(\frac\xi\sigma)$ where
    $\omega\colon\R\to [0,\infty)$, $\omega\in C^\infty$,
    $\supp(\omega)=[-1,1]$, $\int_\R \omega(x)dx=1$.} non-negative
  smooth mollifier with support inside $[-\sigma,\sigma]$. Recall that
  \begin{equation*}
    \abs{\omega_\sigma^{(k)}(\xi)}= \order{\sigma^{-(k+1)}}\ \text{and}\ 
    \int_\R \abs{\xi}^n \abs{\omega_\sigma^{(k)}(\xi)}\,d\xi = \order{\sigma^{n-k}},
  \end{equation*}
  for $k\ge0$ and $n\ge0$.  We then define
  \begin{equation}
    \label{eq:bsigma}
    B_{\sigma}=\{x\in B \mid\text{dist}(x,\partial B)\ge \sigma\}. 
  \end{equation}
  Let $\chi_\sigma$ is a $C^\infty$ function such that
  $\chi_\sigma(\xi)=0$ for $\xi\not\in {B}$, $\chi_\sigma(\xi)=1$ for
  $\xi\in B_\sigma$, $0\le \chi_\sigma(\xi)\le 1$.  Due to the
  smoothness of $\partial B$, we have that
  $\abs{\nabla\chi_\sigma(\xi)}\le C/\sigma$. Furthermore, the
  smoothness of $\partial B$ implies that we can choose $\chi_\sigma$
  such that for every smooth vector field $F$
  \begin{equation}
    \label{eq:F}
    \lim_{\sigma\to0}\int_B F\cdot \nabla \chi_\sigma dx= \int_{\partial B} F\cdot \nu ds.
  \end{equation} 
  
  Then we choose $\test_\eps$ as
  \begin{equation}\label{eq:testK}
    \test_\eps(t,x,s,y)=\omega_{\eps_0}(t-s)W_{\eps_1}(x-y)\chi_{\eps_2}(x)\chi_{\eps_2}(y),
  \end{equation}
  where
 $$\eps=(\eps_0,\eps_1,\eps_2), \qquad W_{\eps_1}(x-y)=\omega_{\eps_1}(x_1-y_1)\cdots \omega_{\eps_1}(x_d-y_d).$$
 With this choice
 \begin{align*}
   \left(\pt+\ps\right)\test_\eps &=0, \\
   \left(\nabla_x + \nabla_y\right)\test_\eps &= 
                                                \omega_{\eps_0}(t-s)W_{\eps_1}(x-y)\left(\nabla_x\chi_{\eps_2}(x)\chi_{\eps_2}(y)+
                                                \chi_{\eps_2}(x)\nabla_y\chi_{\eps_2}(y)\right),
 \end{align*}
 it is straightforward to show that
 \begin{align}
   \lim_{\eps\to 0} i(\test_\eps) &=
                                    \int_B\abs{u_0(x)-v_0(x)}\,dx, \label{eq:ilimit}\\
   \intertext{and} \lim_{\eps\to 0} f(\test_\eps) &=
                                                    \int_B\abs{u(T,x)-v(T,x)}\,dx. \label{eq:flimit}
 \end{align}  
 Next we claim that
 \begin{equation}
   \label{eq:b1limit}
   \lim_{\eps_2\to 0} \left(\lim_{(\eps_0,\eps_1)\to
       0}\abs{B(\test_\eps)}\right) = 0.
 \end{equation}
 To prove this claim first observe that
 \begin{align*}
   B(\test_\eps)&=\int_{\Omega_T^2} \sgn{u-v} \left(k(x)\nabla_x G(u) -
                  k(y)\nabla_y G(v)\right) \\
                &\hphantom{=\int_{\Omega_T^2} }\times \omega_{\eps_0}(t-s)W_{\eps_1}(x-y)
                  \left(\nabla_x\chi_{\eps_2}(x)\chi_{\eps_2}(y) +
                  \chi_{\eps_2}(x)\nabla_y\chi_{\eps_2}(y)\right)\,dX,
 \end{align*}
 and hence
 \begin{align*}
   \lim_{(\eps_0,\eps_1)\to 0}\abs{B(\test_\eps)} &\le
                                                    2\int_{\Omega_T} \abs{(k(x)\nabla_x G(u))\cdot\nabla_x\chi_{\eps_2}(x)}
                                                    \,dxdt \\
                                                  &\quad + 2\int_{\Omega_T} \abs{(k(y)\nabla_y G(v))\cdot\nabla_y\chi_{\eps_2}(y)}
                                                    \,dyds.
 \end{align*}
 Since both $u$ and $v$ satisfy the Neumann boundary conditions (see
 \eqref{eq:F})
 \begin{equation*}
   \lim_{\eps_2\to 0} \int_{B} \abs{(k(x)\nabla_x G(u))\cdot\nabla_x\chi_{\eps_2}(x)}\,dx =0\ \text{a.e.~$t$},
 \end{equation*}
 equation \eqref{eq:b1limit} holds.  Next we tackle the troublesome
 term, cf.~\eqref{eq:goodestimate},
 \begin{equation*}
   C(\test_\eps):=\int_{\Omega_T^2}  \Gamma(u,v)
   \dvxy{\left(h(x)-h(y)\right)^2\test_\eps} \,dX.
 \end{equation*}
 We start by computing
 \begin{align*}
   \dvxy{(h(x)-h(y))^2\test_\eps} &=\dvxy{\begin{matrix}
       (h_1(x)-h_1(y))^2\test_\eps&\cdots&0\\
       \vdots&\ddots&\vdots\\
       0&\cdots&(h_d(x)-h_d(y))^2  \test_\eps 
     \end{matrix}}\\
                                  &=\dvy{\begin{matrix}
                                      \partial_{x_1}  ( (h_1(x)-h_1(y))^2\test_\eps)\\
                                      \vdots\\
                                      \partial_{x_d}(
                                      (h_d(x)-h_d(y))^2 \test_\eps )
                                    \end{matrix}}    \\
                                  &=\sum_{i=1}^d\partial^2_{x_i y_i}( (h_i(x)-h_i(y))^2  \test_\eps )\\
                                  &=-2\sum_{i=1}^d\partial_{x_i}h_i(x)\partial_{y_i}h_i(y) \test_\eps \\
                                  &\quad+2\sum_{i=1}^d(h_i(x)-h_i(y))\partial_{x_i}h_i(x)\partial_{y_i}  \test_\eps \\
                                  &\quad-2\sum_{i=1}^d(h_i(x)-h_i(y))\partial_{y_i}h_i(y)\partial_{x_i}  \test_\eps \\
                                  &\quad+\sum_{i=1}^d(h_i(x)-h_i(y))^2
                                    \partial^2_{x_i y_i} \test_\eps.
 \end{align*}
 Also recall that
 \begin{equation*}
   \Gamma(u,v)=\int_v^u \sgn{\xi-v}
   G'(\xi)\,d\xi \le C\abs{u-v}.
 \end{equation*}
 Consequently
 \begin{align*}
   \abs{C(\test_\eps)}&\le C\underbrace{\sum_{i=1}^d\int_{\Omega_T^2}
                        \abs{u-v}\,\abs{\partial_{x_i}h_i(x)\partial_{y_i}h_i(y)} \test_\eps
                        \,dX}_{C_1(\test_\eps)}\\
                      &\quad + C\underbrace{\sum_{i=1}^d\int_{\Omega_T^2}
                        \abs{u-v}\,\abs{(h_i(x)-h_i(y))\partial_{x_i}h_i(x)}
                        \,\abs{\partial_{y_i}  \test_\eps}\,dX}_{C_2(\test_\eps)}\\
                      &\quad + C\underbrace{\sum_{i=1}^d\int_{\Omega_T^2} \abs{u-v}\,\abs{(h_i(x)-h_i(y))\partial_{y_i}h_i(y)}
                        \,\abs{\partial_{x_i}  \test_\eps}\,dX}_{C_3(\test_\eps)}\\
                      &\quad + C\underbrace{\sum_{i=1}^d\int_{\Omega_T^2} \abs{u-v}\,
                        \left(h_i(x)-h_i(y)\right)^2 \abs{\partial^2_{x_i y_i}\test_\eps}\,dX}_{C_4(\test_\eps)}.
 \end{align*}
 Below we shall repeatedly use that since $v$ is an entropy solution,
 $v$ has a modulus of continuity, i.e., there is a continuous function
 $\nu\colon[0,1]\to [0,\infty)$ such that $\nu(0)=0$ and
 \begin{equation*}
   \sup_{t\in [0,T]} \Bigl(\sup_{\abs{\delta}\le \sigma} \int_{B_\sigma}
   \abs{v(t,x)-v(t,x+\delta)}\,dx\Bigr) \le \nu(\sigma).
 \end{equation*}
 With our choice of test function it readily follows that
 \begin{align*}
   \lim_{\eps_0\to 0}\abs{C_1(\test_\eps)} &\le C\int_{\Omega_T}\int_B
                                             \abs{u-v} W_{\eps_1}(x-y)\,dydxdt\\
                                           &\le
                                             C\int_{\Omega_T}\abs{u(t,x)-v(t,x)}\int_{|y-x|<\eps_1}W_{\eps_1}(x-y)\,dydxdt \\
                                           &\quad + C\int_{\Omega_T}\int_{|y-x|<\eps_1}
                                             \abs{v(t,x)-v(t,y)}W_{\eps_1}(x-y)\,dydxdt.  
 \end{align*}
 Thus
 \begin{equation}
   \label{eq:C1limit}
   \lim_{\eps_0\to 0} \abs{C_1(\test_\eps)} \le C\Bigl(
   \int_{\Omega_T}\abs{u(t,x)-v(t,x)}\,dxdt+\nu(\eps_1)\Bigr).
 \end{equation}
 To estimate $C_2(\test_\eps)$ we observe that
 \begin{equation*}
   \partial_{y_i}\test_\eps = \omega_{\eps_0}(t-s)\left[
     -\partial_{y_i}W_{\eps_1}(x-y)\chi_{\eps_2}(x)\chi_{\eps_2}(y) +
     W_{\eps_1}(x-y)\chi_{\eps_2}(x)\partial_{y_i}\chi_{\eps_2}(y)\right], 
 \end{equation*}
 and we split
 $C_2(\test_\eps)=C_{2,1}(\test_\eps)+C_{2,2}(\test_\eps)$
 accordingly.  Then
 \begin{align*}
   \lim_{\eps_0\to 0} &\abs{C_{2,1}(\test_\eps)}\\
                      &\le \sum_{i=1}^d \int_{\Omega_T}\int_B  \abs{u(t,x)-v(t,y)}\,\abs{h_i(x)-h_i(y)}\\
                      &\qquad\qquad \times    \abs{\partial_{x_i}h_i(x)} \abs{\partial_{y_i}W_{\eps_1}(x-y)}\,
                        \chi_{\eps_2}(x)\chi_{\eps_2}(y)\,dydxdt \\
                      &\le C \int_{\Omega_T}\int_B  \abs{u(t,x)-v(t,y)}\,\abs{x-y} \abs{\nabla_{y}W_{\eps_1}(x-y)}\,dydxdt \\
                      &\le C\int_{\Omega_T}\abs{u(t,x)-v(t,x)} \int_B
                        \abs{x-y}\,\abs{\nabla_{y}W_{\eps_1}(x-y)}\,dydxdt \\
                      &\quad + C\int_{\Omega_T}\int_{|y-x|<\eps_1} \abs{v(t,x)-v(t,y)}
                        \,\abs{x-y}\,\abs{\nabla_{y}W_{\eps_1}(x-y)} \,dydxdt\\
                      &\le C\int_{\Omega_T}\abs{u(t,x)-v(t,x)} \,dxdt 
                        +C\nu(\eps_1).
 \end{align*}
 Similarly we find that
 \begin{align*}
   \lim_{\eps_0\to 0}& \abs{C_{2,2}(\test_\eps)} \\
                     &\le\sum_{i=1}^d 
                       \int_{\Omega_T}\int_B
                       \abs{u(t,x)-v(t,y)}\,\abs{h_i(x)-h_i(y)} \\
                     &\qquad\qquad\times\abs{\partial_{x_i}h(x)}W_{\eps_1}(x-y)
                       \chi_{\eps_2}(x)\abs{\partial_{y_i}\chi_{\eps_2}(y)}\,dydxdt \\
                     &\le C\int_{\Omega_T} \abs{u(t,x)-v(t,x)}\frac{1}{\eps_2}\int_{|y-x|<\eps_1}\abs{x-y}W_{\eps_1}(x-y)\,dydxdt\\
                     & \quad+ \frac{C}{\eps_2}
                       \int_{\Omega_T}\int_{|y-x|<\eps_1}
                       \abs{v(t,x)-v(t,y)}\,\abs{x-y} W_{\eps_1}(x-y)\,dydxdt\\
                     &\le C\frac{\eps_1}{\eps_2} \left(\int_{\Omega_T}
                       \abs{u(t,x)-v(t,x)}\,dxdt + \nu(\eps_1)\right).
 \end{align*}
 Therefore we have that
 \begin{equation}
   \label{eq:C2limit}
   \lim_{\eps_0\to 0}\abs{C_2(\test_\eps)} \le 
   C\left(1+\frac{\eps_1}{\eps_2}\right) \Bigl(
   \int_{\Omega_T}\abs{u(t,x)-v(t,x)}\,dxdt + \nu(\eps_1)\Bigr).
 \end{equation}
 The term $C_3(\test_\eps)$ can similarly be bounded as
 \begin{equation}
   \label{eq:C3limit}
   \lim_{\eps_0\to 0}\abs{C_3(\test_\eps)} \le 
   C\left(1+\frac{\eps_1}{\eps_2}\right) \Bigl(
   \int_{\Omega_T}\abs{u(t,x)-v(t,x)}\,dxdt + \nu(\eps_1)\Bigr).
 \end{equation}
 To estimate the final term $C_4(\test_\eps)$ we first note that
 \begin{align*}
   \partial^2_{x_iy_i}\test_\eps=
   \omega_{\eps_0}(t-s)\Bigl[
   & -\partial^2_{x_iy_i}W_{\eps_1}(x-y)\chi_{\eps_2}(x)\chi_{\eps_2}(y)\\
   &\quad+\partial_{x_i}W_{\eps_1}(x-y)\chi_{\eps_2}(x)\partial_{y_i}\chi_{\eps_2}(y)\\
   &\quad -\partial_{y_i}W_{\eps_1}(x-y)\partial_{x_i}\chi_{\eps_2}(x)\chi_{\eps_2}(y)\\
   &\quad +W_{\eps_1}(x-y)\partial_{x_i}\chi_{\eps_2}(x)\partial_{y_i}\chi_{\eps_2}(y)\Bigr],
 \end{align*}
 and we split
 $C_4(\test_\eps)=C_{4,1}(\test_\eps)+C_{4,2}(\test_\eps)+C_{4,3}(\test_\eps)+C_{4,4}(\test_\eps)$
 accordingly.  These terms are estimated as follows:
 \begin{align*}
   \lim_{\eps_0\to 0}&\abs{C_{4,1}(\test_\eps)}\\
                     &\le\sum_{i=1}^d    \int_{\Omega_T}\int_B \abs{u-v} \left(h_i(x)-h_i(y)\right)^2
                       \abs{\partial^2_{x_iy_i}W_{\eps_1}(x-y)}\chi_{\eps_2}(x)\chi_{\eps_2}(y)\,dydxdt\\
                     &\le C\sum_{i=1}^d    \int_{\Omega_T}\int_B \abs{u-v} \left|x-y\right|^2
                       \abs{\partial^2_{x_iy_i}W_{\eps_1}(x-y)}\,dydxdt\\
                     &\le C\sum_{i=1}^d    \int_{\Omega_T}    \abs{u(t,x)-v(t,x)}\int_{|y-x|<\eps_1}
                       |x-y|^2\abs{\partial^2_{x_iy_i}W_{\eps_1}(x-y)}\,dydxdt\\
                     & \quad+\sum_{i=1}^d \int_{\Omega_T}\int_{|y-x|<\eps_1}
                       \abs{v(t,x)-v(t,y)} |x-y|^2\abs{\partial^2_{x_iy_i}W_{\eps_1}(x-y)}\,dy
                       dxdt\\
                     &\le C\int_{\Omega_T}\abs{u(t,x)-v(t,x)}\,dxdt + C\nu(\eps_1), \\
 \end{align*}
 and
 \begin{align*}
   \lim_{\eps_0\to 0}&\abs{C_{4,2}(\test_\eps)} \\
                     &\le\sum_{i=1}^d
                       \int_{\Omega_T}\int_B \abs{u-v} \left(h_i(x)-h_i(y)\right)^2
                       \abs{\partial_{x_i}W_{\eps_1}(x-y)}\chi_{\eps_2}(x)\abs{\partial_{y_i}\chi_{\eps_2}(y)}\,dydxdt\\
                     &\le C\sum_{i=1}^d  \int_{\Omega_T} \abs{u(t,x)-v(t,x)} \frac{1}{\eps_2}
                       \int_{|y-x|<\eps_1} |x-y|^2\abs{\partial_{x_i}W_{\eps_1}(x-y)}\,dydxdt\\
                     &\quad + C\sum_{i=1}^d\int_{\Omega_T}\abs{v(t,x)-v(t,y)} \frac{1}{\eps_2}\int_{|y-x|<\eps_1}
                       |x-y|^2\abs{\partial_{x_i}W_{\eps_1}(x-y)}\,dydxdt\\
                     &\le C\frac{\eps_1}{\eps_2}\left(\int_{\Omega_T} \abs{u(t,x)-v(t,x)}\,dxdt + \nu(\eps_1)\right). 
 \end{align*}
 The term $C_{4,3}(\test_\eps)$ is similar,
 \begin{equation*}
   \lim_{\eps_0\to 0}\abs{C_{4,3}(\test_\eps)}\le C\frac{\eps_1}{\eps_2}\left(\int_{\Omega_T}
     \abs{u(t,x)-v(t,x)}\,dxdt + \nu(\eps_1)\right).
 \end{equation*}
 It remains the term $C_{4,4}(\test_\eps)$,
 \begin{align*}
   \lim_{\eps_0\to 0}&\abs{C_{4,4}(\test_\eps)} \\
                     &\le\sum_{i=1}^d
                       \int_{\Omega_T}\int_B \abs{u-v} \left(h_i(x)-h_i(y)\right)^2
                       W_{\eps_1}(x-y)\abs{\partial_{x_i}\chi_{\eps_2}(x)\partial_{y_i}\chi_{\eps_2}(y)}\,dydxdt\\
                     &\le C
                       \int_{\Omega_T}\int_B \abs{u-v}  |x-y|^2
                       W_{\eps_1}(x-y)\frac{1}{\eps_2^2}\,dydxdt\\
                     &\le C
                       \int_{\Omega_T} \abs{u(t,x)-v(t,x)} \frac{1}{\eps_2^2}
                       \int_{|y-x|<\eps_1}  |x-y|^2 W_{\eps_1}(x-y)\,dydxdt\\
                     & \quad+ C\int_{\Omega_T}\abs{v(t,x)-v(t,y)} \frac{1}{\eps_2^2}\int_{|y-x|<\eps_1}
                       |x-y|^2 W_{\eps_1}(x-y)\,dydxdt\\
                     &\le C\frac{\eps_1^2}{\eps_2^2}\left(\int_{\Omega_T}
                       \abs{u(t,x)-v(t,x)}\,dxdt + \nu(\eps_1)\right).
 \end{align*}
 Then we can conclude that
 \begin{equation}
   \label{eq:C4limit}
   \lim_{\eps_0\to 0}\abs{C_4(\test_\eps)} \le
   C\Bigl(1+\frac{\eps_1}{\eps_2}+\bigl(\frac{\eps_1}{\eps_2}\bigr)^2\Bigr)\Bigl(
   \int_{\Omega_T}\abs{u(t,x)-v(t,x)}\,dxdt + \nu(\eps_1)\Bigr).
 \end{equation}
 Using \eqref{eq:C1limit}--\eqref{eq:C4limit} we find
 \begin{equation}
   \lim_{\eps_0\to 0}\abs{C(\test_\eps)} \le 
   C\Bigl(1+\frac{\eps_1}{\eps_2}+\bigl(\frac{\eps_1}{\eps_2}\bigr)^2\Bigr)\Bigl(
   \int_{\Omega_T}\abs{u(t,x)-v(t,x)}\,dxdt + \nu(\eps_1)\Bigr), \label{eq:Cest}
 \end{equation}
 and hence we can send $\eps_1$ and $\eps_2$ to zero in such a manner
 that the quotient $\eps_1/\eps_2$ remains finite to conclude that
 \begin{equation}
   \label{eq:c5limit}
   \lim_{\eps\to 0} \abs{C(\test_\eps)}\le C \int_0^T\int_B \abs{u(t,x)-v(t,x)}\,dxdt.
 \end{equation}
 Taking the limit $\eps\to 0$ in \eqref{eq:goodestimate}, using
 \eqref{eq:ilimit}, \eqref{eq:flimit}, \eqref{eq:b1limit} and
 \eqref{eq:c5limit} we get
 \begin{equation*}
   \int_B\abs{u(T,x)-v(T,x)}\,dx\le \int_B\abs{u_0(x)-v_0(x)}\,dx
   + C\int_0^T \int_B \abs{u(t,x)-v(t,x)}\,dxdt.
 \end{equation*}
 The inequality \eqref{eq:stabk} then follows by Gronwall's lemma.
\end{proof}
\begin{remark}\normalfont If $u$ is an entropy solution, then
  $u(t,x)\ge \underline{u}$ a.e.~$(t,x)$ for some finite constant
  $\underline{u}$.  Since
  $\nabla G(u)=\nabla\left(G(u)-G(\underline{u})\right)$, there is no
  loss of generality in assuming $G(\underline{u})=0$. Then define
  \begin{equation*}
    \frak{G}(u)=\int_{\underline{u}}^u G(\sigma)\,d\sigma.
  \end{equation*}
  The $\frak{G}(u(t,x))\ge 0$ for almost all $(t,x)$, and
  \begin{equation*}
    \frac{d}{dt}\int_B \frak{G}(u(t,x))\,dx = -\int_B (k(x)\nabla G(u(t,x)))\cdot \nabla G(u(t,x))\,dx(\le0).
  \end{equation*}
  Since $k_i\ge \underline{k}>0$ this means that
  \begin{align*}
    \int_{\Omega_T} \left|\nabla G(u(t,x))\right|^2\,dxdt &\le \frac{1}{\underline{k}}\int_B
                                                            \frak{G}(u_0(x)) - \frak{G}(u(T,x))\,dx \\&\le  \frac{1}{\underline{k}}\int_B\frak{G}(u_0(x))\,dx<\infty.
  \end{align*}
  We then claim that an entropy solution satisfies the following
  regularity estimate
  \begin{equation}
    \label{eq:Gregular}
    \int_0^T \abs{G(u(t,x))-G(u(t,y))}\,dt \le CT\sqrt{\abs{x-y}},
  \end{equation}
  for some constant $C$ which depends only on $u_0$ and $T$, and for
  almost all $x$ and $y$ in $B$. To substantiate this claim we
  calculate
  \begin{align*}
    \int_0^T&\abs{G(u(t,x))-G(u(t,y))}\,dt \\
            &\le   \int_0^T \Bigl| \int_0^1 \nabla G(u(t,\theta y+(1-\theta)x))\cdot(y-x)\,d\theta\Bigr|\,dt\\
            &\le \sqrt{\abs{x-y}}\, T^{1/2}\Bigl(\int_{\Omega_T} \left|\nabla G(u(t,z))\right|^2\,dzdt\Bigr)^{1/2}.
  \end{align*}
\end{remark}

\section{Existence of a solution}
\label{sec:exist}

Given $\mu>0$, let $\um$ be the unique classical solution of the
initial-boundary value problem \cite[Theorem 1.7.8]{MPS}
\begin{equation}
  \label{eq:regular}
  \begin{cases}
    \pt \um = \dv{k(x)\nabla G(\um)}+\mu\Delta \um,&\quad t>0,\> x\in B,\\
    (k(x)\nabla G(\um(t,x))+\mu\nabla \um(t,x))\cdot\nu=0,&\quad \ t>0,\> x\in\partial B,\\
    \um(0,x)=u_{0}(x),&\quad x\in B.
  \end{cases}
\end{equation}

The main result of this section is the following.

\begin{theorem}[{\bf Existence}]
  \label{th:existence}
  Assume that \ref{ass:B}, \ref{ass:k}, \ref{ass:G}, \ref{ass:init}
  hold and that every $k_i$ depends only on $x_i$.  Then there exists
  an entropy solution $u$ in the sense of
  Definition~\ref{def:entrsolk_new} to the the initial-boundary value
  problem \eqref{eq:boundval}. In particular, we have, as $\mu\to 0$,
  that
  \begin{equation}
    \label{eq:thex}
    \um\to u\quad\text{a.e. and in $L^p(\Omega_T)$ for every $T>0$ and $1\le p<\infty$.}
  \end{equation}
\end{theorem}

First observe that by the Neumann boundary conditions, the solution
operator is conservative, i.e.,
\begin{equation*}
  \frac{d}{dt} \int_B \um(t,x)\,dx = 0.
\end{equation*}

\begin{lemma}
  \label{lm:est1}
  We have that
  \begin{subequations}
  \begin{align}  
    \norm{\um(T,\dott)}_{L^2(B)}^2 + \int_{\Omega_T} 2 G(\um)|h(x)\nabla\um|^2\,dxdt & \notag\\
     +2\mu \int_{\Omega_T} |\nabla\um|^2\,dxdt   &= \norm{u_{0}}_{L^2(B)}^2, \label{eq:L2epsbnd}\\[1mm]
    \label{eq:Linftyepsbnd}
    \inf_{\bar x\in B}u_0(\bar x)\le \um(t,x)&\le \sup_{\bar x\in B} u_0(\bar x), \quad (t,x)\in\Omega\\[1mm]
    \label{eq:L1epsbnd}
    \norm{\um(t,\dott)}_{L^1(B)} &\le \norm{u_0}_{L^1(B)},\\[1mm]
    \label{eq:Gepsxbnd}
    \norm{\nabla G(\um)}_{L^2(\Omega_T)} &\le \frac{\norm{G(u_0)}_{L^\infty(B)}\norm{u_0}_{L^1(B)}}{\sqrt{\min\limits_{i,\, x\in B} k_i(x_i)}},\\[1mm]
    \label{eq:u_tBV}
    \norm{\pt\um(t,\dott)}_{L^1(B)}&\le C\norm{u_0}_{W^{2,1}(B)},\\[1mm]
    \label{eq:divul1}
    \norm{\dv{k\nabla G(\um)}(t,\dott)}_{L^1(B)}&\le C\norm{u_0}_{W^{2,1}(B)}, 
  \end{align}
  \end{subequations}
  for every $\mu>0$ and $t\in[0,\infty)$.
\end{lemma}

\begin{proof}
  If we multiply with $\eta'(\um)$ where $\eta$ is a smooth convex
  function (entropy), using that $\um$ is a classical smooth solution,
  we get
  \begin{equation}
    \label{eq:entrmu}
    \begin{split}
      \pt \eta(\um) &= \eta'(\um)\dv{k(x) \nabla G(\um)}+\mu \eta'(\um)\Delta \um\\
      &=\dv{\eta'(\um)\left(k(x)\nabla G(\um) +\mu \nabla\um\right)}\\
      &\quad - \eta^{\prime\prime}(\um) G'(\um)|h(x)\nabla\um|^2- \mu
      \eta^{\prime\prime}(\um)|\nabla\um|^2.
    \end{split}
  \end{equation}
  Integrating over $B$ and using the boundary conditions we get
  \begin{equation}\label{eq:smooth_entropy}
    \frac{d}{dt}\int_B \eta(\um)\,dx + \int_B  \eta^{\prime\prime}(\um)G'(\um)|h(x)\nabla\um|^2  \,dx 
    + \mu\int_B  \eta^{\prime\prime}(\um)|\nabla\um|^2  \,dx   = 0.
  \end{equation}
  Choosing $\eta(u)=\frac12 u^2$ we get \eqref{eq:L2epsbnd}.  By an
  approximation argument, we can choose $\eta(u)=(u-c)^{\pm}$ where
  $c$ is a constant, to get
  \begin{equation*}
    \int_B (\um(t,x)-c)^{\pm} \,dx \le \int_B \left(u_0-c\right)^{\pm}\,dx.
  \end{equation*}
  This implies the bound \eqref{eq:Linftyepsbnd}.  Setting $c=0$ and
  adding the $(u)^+$ and $(u)^-$ inequalities we get
  \eqref{eq:L1epsbnd}.

  Without loss of generality, we can assume that $G(0)=0.$ If
  necessary by an approximation argument, we can use the entropy
  \begin{equation*}
    \eta(u)=\int_0^u G(v)\,dv.
  \end{equation*}
  Note that $\eta(u)\ge 0$.  Since $\eta'(u)=G(u)$ we get the bound
  \begin{equation*}
    \frac{d}{dt}\int_B \eta(\um)\,dx + \int_B  |h(x)\nabla   G(\um)|^2 \,dx \le 0. 
  \end{equation*}
  This means that $\nabla G(\um)$ is uniformly bounded in
  $L^2(\Omega_T)$ and \eqref{eq:Gepsxbnd} holds.

  Since, from \eqref{eq:regular},
  \begin{equation*}
    (k(x)\nabla \pt G(\um(t,x))+\mu \nabla \pt\um(t,x))\cdot\nu=0,\quad \ t>0,\> x\in\partial B,
  \end{equation*}
  differentiating the equation in \eqref{eq:regular} we gain
  \begin{align*}
    \frac{d}{dt}\int_B |\pt\um|dx&=\int_B \ptt \um \sgn{\pt\um}dx\\
                &=\int_B \dv{(k(x)\nabla \pt G(\um))+\mu \nabla \pt\um}\sgn{\pt\um}dx\\
                &=\underbrace{-\int_B \Big(G'(\um)(k(x)\nabla \pt\um)\cdot\nabla\pt \um +\mu|\nabla\pt \um|^2\Big)\text{sign}'(\pt\um)dx}_{\le0}\\
                &\quad\underbrace{-\int_B G''(\um)\pt\um (k(x)\nabla \um)\cdot\nabla \pt\um\text{sign}'(\pt\um)dx}_{=0}\\
                &\quad\underbrace{+\int_{\partial B} (k(x)\nabla \pt G(\um)+\mu\nabla\pt\um)\cdot\nu \sgn{\pt\um}ds.}_{=0}
  \end{align*}
  The middle term disappears as the integrand contains a term of the
  form $\psi\, \text{sign}'(\psi)$ (with $\psi=\pt\um$).  Integrating
  over $(0,t)$ we get \eqref{eq:u_tBV} and from the equation
  \eqref{eq:divul1}.
\end{proof}

\begin{lemma}
  \label{lm:BVx}
  For every $\sigma>0$ there exists a constant $C_\sigma>0$
  independent of $\mu$ such that
  \begin{equation}
    \label{eq:BVx}
    \norm{\partial_{x_j} \um(t,\dott)}_{L^1(B_\sigma)}\le \norm{\partial_{x_j} u_0}_{L^1(B)}e^{ C_\sigma t},
  \end{equation}
  for every $\mu>0$, $t\ge0$, and $j\in\{1,\dots,d\}$, where
  $B_\sigma$ is defined in \eqref{eq:bsigma}.
\end{lemma}

\begin{proof}
  Define
  \begin{equation*}
    \vm_i=\partial_{x_i}\um.
  \end{equation*}
  The equation \eqref{eq:regular} reads
  \begin{equation*}
    \pt \um = \sum_{i=1}^d\partial_{x_i}\Big( k_i(x_i)G'(\um)\vm_i\Big)+\mu\Delta \um.
  \end{equation*}
  Differentiating with respect to $x_j$ we get
  \begin{align*}
    \pt \vm_j &= \sum_{i=1}^d\partial_{x_i}\Big( k_i(x_i)G'(\um)\partial_{x_i}\vm_j\Big)
                +\sum_{i=1}^d\partial_{x_i}\Big( k_i(x_i)G''(\um)\vm_i\vm_j\Big)\\
              &\quad +\sum_{i=1}^d\partial_{x_i}\Big((\partial_{x_j} k_i(x_i))G'(\um)\vm_i\Big) +\mu\Delta \vm_j.
  \end{align*}
  Let $\chi$ be a cut-off function such that
  \begin{align*}
   & \chi\in C^\infty(\R^d),\qquad 0\le\chi\le1,\qquad\chi(x)=\begin{cases}
      1,&\text{if $x\in B_\sigma$},\\
      0,&\text{if $x\not\in B_{\frac{\sigma}{2}}$},
    \end{cases}\\
  &  \left|\partial^2_{x_i x_j}\chi\right|\le c(\sigma)\chi,\quad\left|\partial_{x_i}\chi\right|\le c(\sigma)\chi.
  \end{align*}
  Using the facts
  \begin{equation*}
    \chi\big|_{\partial B}=\partial_{x_i}\chi\big|_{\partial B}=0,\qquad i\not =j\Rightarrow \partial_{x_j}k_i=0,
  \end{equation*}
  we have that
  \begin{align*}
    \frac{d}{dt}\int_{B}|\vm_j|\chi dx &=\int_{B}\pt \vm_j \sign{\vm_j}\chi dx\\
                                       &= \sum_{i=1}^d\int_{B}\partial_{x_i}\Big( k_i(x_i)G'(\um)\partial_{x_i}\vm_j\Big)\sgn{\vm_j}\chi dx\\
                                       &\quad +\sum_{i=1}^d\int_{B}\partial_{x_i}\Big( k_i(x_i)G''(\um)\vm_i\vm_j\Big)\sgn{\vm_j}\chi dx\\
                                       &\quad +\int_{B}\partial_{x_j}\Big((\partial_{x_j} k_j(x))G'(\um)\vm_j\Big)\sgn{\vm_j}\chi dx\\
                                       &\quad +\mu\int_{B}\Delta \vm_j\sgn{\vm_j}\chi dx\\
                                       &=\underbrace{- \sum_{i=1}^d\int_{B} k_i(x_i)G'(\um)(\partial_{x_i}\vm_j)^2\text{sign}'(\vm_j)\chi dx}_{\le0}\\
                                       &\quad -\sum_{i=1}^d\int_{B} k_i(x_i)G'(\um)\partial_{x_i}|\vm_j|\partial_{x_i}\chi dx\\
                                       &\quad\underbrace{-\sum_{i=1}^d\int_{B}\partial_{x_i}k_i(x_i)G''(\um)\vm_i\vm_j(\partial_{x_i}\vm_j)\text{sign}'(\vm_j)\chi dx}_{=0}\\
                                       &\quad -\sum_{i=1}^d\int_{B}k_i(x_i)G''(\um)\vm_i|\vm_j|\partial_{x_i}\chi dx\\
                                       &\quad\underbrace{-\int_{B}(\partial_{x_j} k_j(x))G'(\um)\vm_j(\partial_{x_j}\vm_j)\text{sign}'(\vm_j)\chi dx}_{=0}\\
                                       &\quad -\int_{B}(\partial_{x_j} k_j(x))G'(\um)|\vm_j|\partial_{x_j}\chi dx\\
                                       &\quad \underbrace{-\mu\int_{B}|\nabla \vm_j|^2\text{sign}'(\vm_j)\chi dx}_{=0}
                                         +\mu\int_{B} |\vm_j|\Delta\chi dx \\
                                       &\le -\sum_{i=1}^d\int_{B} k_i(x_i)\partial_{x_i}\Big(G'(\um)|\vm_j|\Big)\partial_{x_i}\chi dx\\
                                       &\quad -\int_{B}(\partial_{x_j} k_j(x))G'(\um)|\vm_j|\partial_{x_j}\chi dx  +\mu\int_{B} |\vm_j|\Delta\chi dx\\
                                       &= \sum_{i=1}^d\int_{B} (\partial_{x_i}k_i(x_i))G'(\um)|\vm_j|\partial_{x_i}\chi dx\\
                                       &\quad +\sum_{i=1}^d\int_{B} k_i(x_i)G'(\um)|\vm_j|\partial_{x_i x_i}^2\chi dx\\
                                       &\quad -\int_{B}(\partial_{x_j} k_j(x))G'(\um)|\vm_j|\partial_{x_j}\chi dx  +\mu\int_{B} |\vm_j|\Delta\chi dx\\
                                       &\le c(\sigma)\int_B|\vm_j|\chi dx.
  \end{align*}
  The term $\int_{B}|\nabla \vm_j|^2\text{sign}'(\vm_j)\chi dx$
  vanishes due to \cite[Lemma B.5]{MR3443431}.  The Gronwall lemma
  gives the claim.
\end{proof}

\begin{proof}[Proof of Theorem \ref{th:existence}]
  We claim that there exists a function $u\in L^\infty(\Omega)$ and a
  subsequence $\{\mu_{\ell}\}_{\ell},\,\mu_{\ell}\to0,$ such that
  \begin{equation}
    \label{eq:ex1}
    u^{\mu_{\ell}}\to u,\qquad \text{a.e. and in $L^p(\Omega_T)$ for every $T>0$ and $1\le p<\infty$.}
  \end{equation}
  Consider the sequence $\{\um \chi_{B_\sigma}\}_{\mu}$. Thanks to
  Lemmas \ref{lm:est1} and \ref{lm:BVx} for every $\sigma$ we can find
  a sequence $\{\mu_{\ell}^\sigma\}_{\ell},\,\mu_{\ell}^\sigma\to0,$
  and a function $u_\sigma\in BV(\Omega_T)$ such that
  \begin{equation*}
    u^{\mu^\sigma_{\ell}}\to u_\sigma,\quad \text{a.e. and in $L^p((0,T)\times B_\sigma)$ for every $T>0$ and $1\le p<\infty$.}
  \end{equation*}
  Since
  \begin{equation*}
    u^{\mu^\sigma_{\ell}}=u^{\mu^{\sigma'}_{\ell}}\quad \text{in $B_\sigma$, if $\sigma'<\sigma$.}
  \end{equation*}
  The function $u_\sigma$ can be extended to a function
  $u\in L^\infty(\Omega)$ such that
  \begin{equation*}
    u_\sigma=u\quad \text{in $B_\sigma$.}
  \end{equation*}
  and using a diagonal argument we can find a subsequence
  $\{\mu_{\ell}\}_{\ell},\,\mu_{\ell}\to0,$ such that \eqref{eq:ex1}
  holds.

  The dominated convergence theorem, Lemmas \ref{lm:est1},
  \ref{lm:BVx}, \eqref{eq:entrmu}, and \eqref{eq:ex1} guarantee that
  \eqref{eq:visck}, \eqref{eq:visc1k}, \eqref{eq:defweak},
  \eqref{eq:entrknew} hold. Regarding \eqref{eq:visc2k}, we observe
  that thanks to \eqref{eq:u_tBV}
  \begin{equation*}
    \dv{k(x)\nabla G(u)}=\pt u \in L^\infty((0,\infty);L^1(B)).
  \end{equation*}
  In light of \cite[Theorem 2.1]{Chen2001}, we have that
  $(k\nabla G(u))\cdot\nu$ admits trace on
  $(0,\infty)\times\partial B$ and \eqref{eq:visc2k} holds.

  Finally, we can improve the convergence along a subsequence in
  \eqref{eq:ex1} to the one along all the family in \eqref{eq:thex}
  due to the uniqueness of the entropy solutions.
\end{proof}

\section{Convergence of a difference scheme in one space dimension}
\label{sec:num}

We now consider the problem in one space dimension with $B=(0,1)$. Thus
\begin{equation*}
\pt u = \big(k(x)\px G(u)\big)_x, \quad u|_{t=0}=u_0,
\end{equation*}
with Neumann boundary conditions $\big(k(x)\px G(u)\big)|_{x=0, 1} = 0$.

Let $\Dx=1/(N+1)$ for some positive integer $N$. We use the notation
\begin{equation*}
  \Delta_{\pm}a_j=\pm\left(a_{j\pm 1}-a_j\right),
\end{equation*}
and $x_j = (j+1/2)\Dx$, $t_n=n\Dt$, and $k_{j+1/2}=k(x_{j+1/2})$ for
(small) positive numbers $\Dx$ and $\Dt$.  With the shift operator
\begin{equation*}
  (S^\pm a)_j=a_{j\pm1},
\end{equation*}
we can write
\begin{equation*}
  S^\pm  \Delta_{\mp}=\Delta_{\pm}.
\end{equation*}
We use the common convention that $u^n_j$ is an approximation to
$u(t_n,x_j)$. Let $u^n_j$, $j=0,\ldots,N$, $n\ge 0$ be the solution to
the following system of equations
\begin{equation}
  \label{eq:implicit}
  \unpj - \mu \Dp\left(k_{j-1/2}\Dm G\left(\unpj\right)\right) = \unj, \ j=0,\ldots,N,
\end{equation}
with the boundary conditions that $\Dm u^{n+1}_0=\Dp u^n_N=0$. Here
$\mu=\Dt/\Dx^2$ (which is only assumed to be bounded from below). As
to the initial condition we define
\begin{equation}
  \label{eq:diff_initialcond}
  u^0_j=\frac{1}{\Dx} \int_{x_{j-1/2}}^{x_{j+1/2}} u_0(x)\,dx,\ \ j=0,\ldots,N.
\end{equation}
For later use we also define the function $u_{\Dt}\colon [0,\infty)\times B\to \R$ as
\begin{equation}
  \label{eq:diff_uDt}
  u_\Dt(t,x)=u^n_j \ \text{for}\ (t,x)\in [t_n,t_{n+1})\times [x_{j-1/2},x_{j+1/2}).
\end{equation}
Regarding the solvability of \eqref{eq:implicit}, we have the
following result.
\begin{lemma} For any given $u^n=(u^n_0,\dots,u^n_N)\in\R^{N+1}$, the equation
  \eqref{eq:implicit} has a unique solution $u^{n+1}\in\R^{N+1}$.
\end{lemma}
\begin{proof}
  Let $\alpha=(\alpha_0,\dots,\alpha_N)\in\R^{N+1}$ be given.
  Consider the map $\R^{N+1} \ni u \mapsto F(u)\in \R^{N+1}$, where
  \begin{equation*}
    F(u)_j = u_j - \mu \Dp\left(\kjm \Dm G(u_j)\right)-\alpha_j.
  \end{equation*}
  Then the statement of the lemma is equivalent to the existence of a
  unique $u$ such that $F(u)=0$.  We have
  \begin{align*}
    \big(F(u),u\big)&= \norm{u}^2 -\mu\sum_{j=0}^{N+1}\Dp\left(\kjm \Dm G(u_j)\right)u_j -\big(\alpha, u \big)\\
                    &=\norm{u}^2 +\mu\sum_{j=0}^{N+1}\kjm \Dm G(u_j)\Dm u_j- \big(\alpha, u \big)\\
                    &=\norm{u}^2 +\mu\sum_{j=0}^{N+1}\kjm G'(\tilde u_j)|\Dm u_j|^2 -\big(\alpha, u \big)\\
                    &\ge \big(\norm{u}- \norm{\alpha}\big)\norm{u}.
  \end{align*}
  Here $\tilde u_j$ is a number between $u_j$ and $u_{j+1}$, and we
  have used the monotonicity of $G$.  For a given $\alpha$ with norm
  $r=\norm{\alpha}$, we have that $\big(F(u),u\big)\ge0$ for all $u$
  with $\norm{u}\ge r$. Then \cite[Thm. 9.9-3]{MR3136903} gives the
  existence of an $\bar u$ such that $F(\bar u)=0$.  Uniqueness
  follows from \eqref{eq:entydig_stab}.  More precisely, we consider
  two solutions
  \begin{equation*}
    \begin{gathered}
      \left.
        \begin{aligned}
          u_j - \mu\Dp\left(\kjm\Dm G(u_j)\right) &=f_j\\
          v_j - \mu\Dp\left(\kjm\Dm G(v_j)\right) &=g_j
        \end{aligned}
      \right\} \ \text{for $j=0,\ldots,N$},
    \end{gathered}
  \end{equation*}
  with the boundary conditions $\Dm u_0=\Dp u_N = 0$ and
  $\Dm v_0=\Dp v_N = 0$.  After a detailed analysis of each term we
  conclude (cf.~\eqref{eq:entydig_stab})
  \begin{equation}\label{eq:entydig_stab_0}
    \sum_{j=0}^N \abs{u_j-v_j} \le \sum_{j=0}^N \abs{f_j-g_j}.
  \end{equation}

\end{proof}

It will also be useful to define the (artificial) value $u^{-1}$ by
taking an explicit step in the negative $t$ direction,
\begin{equation*}
  u^{-1}_j = u^0_j - \mu \Dp\left(k_{j-1/2}\Dm G(u^0_j)\right),\  j=0,\ldots,N,
\end{equation*}
so that the scheme \eqref{eq:implicit} is valid for $n\ge -1$.

 
Due to the boundary condition, the scheme is conservative,
\begin{align*}
  \sum_{j=0}^N \unj &= \sum_{j=0}^N \unpj - \mu \sum_{j=0}^N \Dp\left(k_{j-1/2}\Dm
                      G\left(\unpj\right)\right)\\
                    &=\sum_{j=0}^N \unpj + \mu\left(k_{-1/2}\Dm G(u^{n+1}_0) - k_{N+1/2}\Dp
                      G(u^{n+1}_N)\right)\\
                    &=\sum_{j=0}^N \unpj.
\end{align*}
If $\eta(u)$ is a twice continuously differentiable function,
\begin{equation*}
  \eta'\left(\unpj\right)\left(\unpj-\unj\right) = 
  \eta\left(\unpj\right)-\eta(\unj) + \frac12 \eta''\left(u^{n+1/2}_j\right)\left(\unpj-\unj\right)^2,
\end{equation*}
where $u^{n+1/2}_j$ is some value between $\unpj$ and $\unj$. We can
then multiply the scheme \eqref{eq:implicit} by $\eta'(\unpj)$ to find
that
\begin{equation*}
  \eta\left(\unpj\right) - \mu \eta'(\unpj)\Dp\left(k_{j-1/2}\Dm G\left(\unpj\right)\right)
  + \frac12 \eta''\left(u^{n+1/2}_j\right)\left(\unpj-\unj\right)^2 
  =\eta(\unj).
\end{equation*}
Using the ``Leibniz rule''
$\Dp \left(a_j b_j\right) = a_j\Dp b_j + b_{j+1}\Dp a_j$, this can be
rewritten
\begin{multline}\label{eq:discpointentr}
  \eta\left(\unpj\right) - \mu \Dp\left[\kjm\eta'(\unpj)\Dm
    G\left(\unpj\right)\right]\\
  +\mu \kjp\left(\Dp \eta'(\unpj)\right)\left(\Dp G(\unpj)\right)
  + \frac12 \eta''\left(u^{n+1/2}_j\right)\left(\unpj-\unj\right)^2 \\
  =\eta(\unj).
\end{multline}
If we sum this over $j$, then, due to the boundary conditions, the
second term on the left vanishes, and we are left with
\begin{multline}
  \label{eq:discentr}
  \sum_{j=0}^N \eta(\unpj) + \mu \sum_{j=0}^N \kjp\left(\Dp
    \eta'(\unpj)\right)\left(\Dp
    G(\unpj)\right) \\
  +\frac12 \sum_{j=0}^N
  \eta''\left(u^{n+1/2}_j\right)\left(\unpj-\unj\right)^2
  =\sum_{j=0}^N\eta(\unj).
\end{multline}
Setting $\eta(u)=\frac12 u^2$, we get the ``$L^2$'' bound
\begin{align*}
  \frac12 \sum_{j=0}^N \left(\unpj\right)^2 &+ \mu\sum_{j=0}^N \kjp\left(\Dp
                                              \unpj\right)
                                              \left(\Dp G(\unpj)\right)\\
                                            &\qquad  + \frac12\mu^2\sum_{j=0}^N \left(\Dp\big(\kjp\Dm
                                              G(\unpj)\big)\right)^2
                                              =\frac12 \sum_{j=0}^N \left(\unj\right)^2.
\end{align*}
Summing this over $n=0,\ldots,M-1$ we find
that 
\begin{align*}
  \sum_{j=0}^N\left( u^M_j\right)^2 &+
                                      2 \mu\sum_{n=0}^{M-1} \sum_{j=0}^N \kjp\left(\Dp
                                      \unj\right)
                                      \left(\Dp G(\unj)\right) \\
                                    &\qquad + \mu^2  \sum_{n=0}^{M-1}\sum_{j=0}^N\left(\Dp\big(\kjp\Dm
                                      G(\unj)\big)\right)^2
                                      = \sum_{j=0}^N \left(u^0_j\right)^2.
\end{align*}
In particular, this implies the uniform $L^2$ bound on $u_\Dt$,
\begin{equation}
  \label{eq:discl2}
  \norm{u_\Dt(t,\dott)}_{L^2(B)} \le \norm{u_0}_{L^2(B)}. 
\end{equation}
We can also choose $\eta(u)=\abseps{u}$ in \eqref{eq:discentr}, and
then let $\eps\to 0$ to conclude that
\begin{equation*}
  \norm{u_\Dt(t,\dott)}_{L^1(B)} \le \norm{u_0}_{L^1(B)}.
\end{equation*}
If we choose $\eta(u)=\int_0^u G(v)\,dv$ in equation
\eqref{eq:discentr} we find
\begin{equation*}
  \sum_{j=0}^N \eta(\unpj) + \mu \sum_{j=0}^N \kjp \left(\Dp
    G\left(\unpj\right)\right)^2 
  \le \sum_{j=0}^N \eta(\unj).
\end{equation*}
If $\Dx\sum_j \eta(u^0_j)$ is uniformly bounded, then
\begin{equation}
  \label{eq:Dgl2bnd}
  \Dt\Dx \sum_{n=0}^M \sum_{j=0}^N\kjp \Bigl(\frac{\Dp G(\unpj)}{\Dx}\Bigr)^2 \le C.
\end{equation}

Next we show stability of solutions to \eqref{eq:implicit} with
respect to the initial data. To keep the notation simple, let $u$ and
$v$ solve
\begin{equation*}
  \begin{gathered}
    \left.
      \begin{aligned}
        u_j - \mu\Dp\left(\kjm\Dm G(u_j)\right) &=f_j\\
        v_j - \mu\Dp\left(\kjm\Dm G(v_j)\right) &=g_j
      \end{aligned}
    \right\}
    \ \text{for $j=0,\ldots,N$},\\
    \text{with the boundary conditions}\ \left\{
      \begin{aligned}
        \Dm u_0&=\Dp u_N = 0,\\
        \Dm v_0&=\Dp v_N = 0.
      \end{aligned}
    \right.
  \end{gathered}
\end{equation*}
Subtracting these equations
\begin{equation*}
  \left(u_j-v_j\right) - \mu\Dp\left(\kjm \Dm
    \left(G(u_j)-G(v_j)\right)\right)
  =f_j-g_j.
\end{equation*}
Multiplying with $\signe(u_j-v_j)$
\begin{align*}
  \signe(u_j-v_j) \left(u_j-v_j\right) &- \mu\Dp\left[
    \kjm\signe(u_j-v_j) \Dm\left(G(u_j)-G(v_j)\right)\right]
  \\
  &+\mu \kjp \left[\Dp
    \signe(u_j-v_j)\right]\left[\Dp\left(G(u_j)-G(v_j)\right)\right]\\
  &\qquad\qquad\qquad\qquad \qquad\qquad =  \signe(u_j-v_j)\left(f_j-g_j\right).
\end{align*}
Summing over $j=1,\ldots,N$ and sending $\eps\to 0$,
\begin{equation*}
  \sum_{j=0}^N \abs{u_j-v_j} + \mu\sum_{j=0}^N \kjp
  \left[\Dp
    \sgn{u_j-v_j}\right]\left[\Dp\left(G(u_j)-G(v_j)\right)\right]\le 
  \sum_{j=0}^N \abs{f_j-g_j}.
\end{equation*}
Consider the second sum on the left, each summand reads
\begin{equation*}
  \kjp 
  \left[\sgn{u_{j+1}-v_{j+1}}-\sgn{u_j-v_j}\right] 
  \left[G(u_{j+1})-G(v_{j+1})-(G(u_j)-G(v_j))\right].
\end{equation*}
We have that
\begin{equation*}
  \sgn{u_{j+1}-v_{j+1}}-\sgn{u_j-v_j} = 
  \begin{cases}
    2 &\text{if $u_{j+1}>v_{j+1}$ and $u_j<v_j$,}\\
    1 &\text{if $u_{j+1}>v_{j+1}$ and $u_j=v_j$,}\\
    1 &\text{if $u_{j+1}=v_{j+1}$ and $u_j<v_j$,}\\
    -1&\text{if $u_{j+1}=v_{j+1}$ and $u_j>v_j$,}\\
    -1&\text{if $u_{j+1}<v_{j+1}$ and $u_j=v_j$,}\\
    -2&\text{if $u_{j+1}<v_{j+1}$ and $u_j>v_j$,}\\
    0&\text{otherwise.}
  \end{cases}
\end{equation*}
Similarly we find that
\begin{equation*}
  G(u_{j+1})-G(v_{j+1})-(G(u_j)-G(v_j))
  \begin{cases}
    \ge 0 &\text{if $u_{j+1}>v_{j+1}$ and $u_j<v_j$,}\\
    \ge 0 &\text{if $u_{j+1}>v_{j+1}$ and $u_j=v_j$,}\\
    \ge 0 &\text{if $u_{j+1}=v_{j+1}$ and $u_j<v_j$,}\\
    \le 0&\text{if $u_{j+1}=v_{j+1}$ and $u_j>v_j$,}\\
    \le 0&\text{if $u_{j+1}<v_{j+1}$ and $u_j=v_j$,}\\
    \le 0&\text{if $u_{j+1}<v_{j+1}$ and $u_j>v_j$.}\\
  \end{cases}
\end{equation*}
Thus each summand in the second sum over $j$ above is nonnegative and
we find that
\begin{equation} \label{eq:entydig_stab} \sum_{j=0}^N \abs{u_j-v_j}
  \le \sum_{j=0}^N \abs{f_j-g_j}.
\end{equation}
This completes the argument used to prove \eqref{eq:entydig_stab_0}.

If $v^n_j$ is another solution of \eqref{eq:implicit}, with initial
data $v^0_j$, then we have that
\begin{equation}
  \label{eq:discl1cont}
  \sum_{j=0}^N \abs{\unj-v^n_j}\le \sum_{j=0}^N \abs{u^0_j-v^0_j}.
\end{equation}
Since the update $u^n\mapsto u^{n+1}$ is conservative, by the
Crandall--Tartar lemma \cite[Lemma 2.13]{MR3443431}, it is also
monotone, i.e., if $u^n_j\le v^n_j$, then $u^{n+1}_j\le v^{n+1}_j$. If
we set $v^n_j=u^{n+1}_j$ we get the estimate
\begin{equation*}
  \sum_{j=0}^N \abs{\unpj-\unj} \le \sum_{j=0}^N \abs{u^0_j -
    u^{-1}_j}=\mu\sum_{j=0}^N \abs{\Dp\left(\kjm\Dm G(u^0_j)\right)}.
\end{equation*}
Now we assume that $u_0$ is such that $k \px G(u_0)\in BV(B)$.  This
then gives the estimates
\begin{align}
  \label{eq:discL1cont}
  &\begin{aligned}
    \norm{u_\Dt(t+\Dt,\dott)-u_\Dt(t,\dott)}_{L^1(B)} &\le C\Dt \int_B
    \abs{\px(k\px G(u_0(x)))}\,dx\\ 
    &= C\Dt \abs{k\px G(u_0)}_{BV(B)},
  \end{aligned}\\
  \intertext{and} &\Dx\sum_{j=0}^N \Bigl|\frac{\Dp\left(\kjm\Dm
                    G(u^n_j)\right)}{\Dx^2}\Bigr| \le C\abs{k\px
                    G(u_0)}_{BV(B)}\label{eq:Gxbvdisc}.
\end{align}
As for the viscous regularization \eqref{eq:regular}, we now establish
a $BV$ bound on $u_\Dx$. Define $v^n_{j-1/2}=\Dm u^n_j$. Then
$v^n_{-1/2}=v^n_{N+1/2}=0$, and $v^n_{j-1/2}$ solves the equation
\begin{equation}
  \label{eq:vdef}
  \vnmj - \mu\Dm\Dp\left(\kjm\Gamma_{j-1/2}^{n+1} \vnmj\right) = v^n_{j-1/2}, \quad j=1,\ldots,N,
\end{equation}
where
\begin{equation*}
  \Gamma_{j-1/2}^{n+1} = \frac{\Dm G(\unpj)}{\Dm \unpj}= \frac{\Dm G(\unpj)}{\vnmj}.
\end{equation*}
Set $\eta(v)=\abseps{v}$, multiply with $\eta'(v_{j-1/2}^{n+1})$ to
get
\begin{multline*}
  \eta(v_{j-1/2}^{n+1}) - \mu
  \Dp\left[\eta'(v_{j-1/2}^{n+1})\Dm\left(\kjm\Gamma_{j-1/2}^{n+1}
      v_{j-1/2}^{n+1}\right)\right]\\
  -\mu\kjp\left[\Dp
    \eta'(v_{j-1/2}^{n+1})\right]\left[\Dp\left(\kjm\Gamma_{j-1/2}^{n+1}v_{j-1/2}^{n+1}\right)\right]\le
  \eta(v_{j-1/2}^{n}).
\end{multline*}
Here we have used that
\begin{equation*}
  \eta(v_{j-1/2}^{n})- \eta(v_{j-1/2}^{n+1}) \ge \big(v_{j-1/2}^{n}-v_{j-1/2}^{n+1})\big)\eta'(v_{j-1/2}^{n+1})
\end{equation*}
due to the convexity of $\eta$, and
\begin{align*}
  \eta'(v_{j-1/2}^{n+1})\Dm\Dp& \left(\kjm\Gamma_{j-1/2}^{n+1} \vnmj\right)  \\
                              &\quad =\Dp\Big[ \eta'(v_{j-1/2}^{n+1})\Dm\left(\kjm\Gamma_{j-1/2}^{n+1} \vnmj\right) \Big] \\
                              &\qquad  - \big(\Dp\eta'(v_{j-1/2}^{n+1})\big)S^+\Big(\Dm\left(\kjm\Gamma_{j-1/2}^{n+1} \vnmj\right)\Big) \\
                              &\quad= \Dp\Big[ \eta'(v_{j-1/2}^{n+1})\Dm\left(\kjm\Gamma_{j-1/2}^{n+1} \vnmj\right) \Big]\\
                              &\qquad
                                - \big(\Dp\eta'(v_{j-1/2}^{n+1})\big)\Dp\left(\kjm\Gamma_{j-1/2}^{n+1} \vnmj\right).
\end{align*}
Now we sum this over $j=0,\ldots,N$, the second term on the left
vanishes due to the boundary condition on $v_{j-1/2}^{n+1}$, and the
fact that $\eta'(0)=0$. Then we can let $\eps\to 0$ to conclude that
\begin{equation*}
  \sum_{j=0}^{N-1} \abs{\vnpj} + \mu\sum_{j=1}^{N} 
  \kjm\left[\Dp
    \sgn{\vnmj}\right]\left[\Dp\left(\Gamma_{j-1/2}\vnmj\right)\right]
  \le \sum_{j=0}^{N-1} \abs{\vnj}.
\end{equation*}
Now
\begin{equation*}
  \sgn{\vnpj}-\sgn{\vnmj}=
  \begin{cases}
    2 & \vnmj<0<\vnpj,\\
    1 & \vnmj<0=\vnpj,\\
    1 & \vnmj=0<\vnpj,\\
    -1& \vnmj=0>\vnpj,\\
    -1& \vnmj>0=\vnpj,\\
    -2& \vnmj>0>\vnpj,\\
    0 &\text{otherwise,}
  \end{cases}
\end{equation*}
and
\begin{equation*}
  \Gamma_{j+1/2}^{n+1}\vnpj - \Gamma_{j-1/2}^{n+1} \vnmj
  \begin{cases}
    \ge 0& \vnmj<0<\vnpj,\\
    \ge 0& \vnmj<0=\vnpj,\\
    \ge 0& \vnmj=0<\vnpj,\\
    \le 0& \vnmj=0>\vnpj,\\
    \le 0& \vnmj>0=\vnpj,\\
    \le 0& \vnmj>0>\vnpj.
  \end{cases}
\end{equation*}
Thus the second sum above is nonnegative, and we conclude that
\begin{equation}
  \label{eq:discBV}
  \sum_{j=0}^{N-1} \abs{u^n_{j+1}-\unj}\le \sum_{j=0}^{N-1}
  \abs{u^0_{j+1}-u^0_j}, \ n\ge 0.
\end{equation}
We have established that $\seq{u_{\Dt}}_{\Dt>0}$ is uniformly bounded
in $L^\infty(\Omega_T)$ and also $C([0,T];L^1(B))$, and that we have a
uniform (in $t$ and $\Dt$) bound on $\abs{u_\Dt(t,\dott)}_{BV(B)}$. By
Kolmogorov--Riesz's theorem (see \cite[Thm.~A.11]{MR3443431},
\cite{MR2734454,MR3964217}), there exists a subsequence of $\Dt$'s and
a function $u=u(t,x)$, such that $u_\Dt \to u$ in
$C([0,T];L^1(B))$. Since $G$ is a Lipschitz continuous function, also
$G(u_\Dt)$ converges to $G(u)$. We shall have use for the notation
\begin{equation*}
  D_{\pm}=\frac{1}{\Dx}\Delta_{\pm} \ \text{and}\ \dpt
  \eta\left(\unj\right) = \frac{\eta(\unpj)-\eta(\unj)}{\Dt}.
\end{equation*}
Define the continuous function $k DG_\Dx$ by piecewise linear
interpolation
\begin{equation*}
  kDG_\Dx(t,x)=\kjm \ddm G(u_\Dt(t,x_j)) + \frac{x-x_{j-1/2}}{\Dx} \ddp\left(\kjm\ddm
    G(u_\Dt(t,x_j))\right), 
\end{equation*}
for $x\in (x_{j-1/2},x_{j+1/2}]$. Then, for fixed $t$, by the estimate
\eqref{eq:Gxbvdisc}, $kDG_\Dx(t,\dott)$ is in $BV(B)$ (with a
uniformly bounded $BV$ seminorm). Hence there is a (further)
subsequence of $\Dt$'s such that $kDG_\Dt(t,\dott)\to H(t,\dott)$ in
$L^1(B)$. Furthermore $H=k \px G(u)$ weakly in $B$. This also implies
that $\px G(u)\in L^\infty([0,T]; BV(B))$. The bound
\eqref{eq:Dgl2bnd} implies that $\px G(u)$ is in $L^2(\Omega_T)$. In
order to prove that $u$ is an entropy solution, we start with the
discrete entropy inequality \eqref{eq:discpointentr}, which implies
that
\begin{multline*}
  \dpt \eta\left(\unj\right) - \ddp\left(\eta'(\unpj)\kjm \ddm
    G(\unpj)\right) \\
  + \kjp\left(\ddp \eta'(\unpj)\right)\left(\ddp G(\unpj)\right) \le
  0.
\end{multline*}
Let $\test$ be a suitable nonnegative test function and set
\begin{equation*}
  \test^n_j = \frac{1}{\Dt\Dx}\int_{t_n}^{t_{n+1}}
  \int_{x_{j-1/2}}^{x_{j+1/2}} \test(s,y)\,dyds.
\end{equation*}
We multiply the above inequality with $\Dt\Dx\test^n_j$ and do
summation by parts in $n$ and $j$ to get
\begin{align}
  \Dt\Dx \sum_{n=0}^M\sum_{j=1}^N &\eta^n_j D^t_- \test^n_j -
                                    \Dx\sum_{j=1}^N \eta^{M+1}_j\test^{M}_j  + \Dx\sum_{j=0}^N \eta^0_j
                                    \test^0_j\label{eq:tpart}\\
                                  &\quad -\Dt\Dx\sum_{n=0}^M \sum_{j=0}^N \big(\eta^\prime\big)^{n+1}_j \kjm \ddm G^{n+1}_j
                                    \ddm\test^{n}_j \label{eq:convpart}\\
                                  &\quad -\Dt\Dx \sum_{n=0}^M \sum_{j=0}^N \kjp\left(\ddp
                                    \big(\eta^\prime\big)^{n+1}_j\right)\left(\ddp G^{n+1}_j\right)\test^n_j \ge
                                    0, \label{eq:squarepart}
\end{align}
where $\eta^{n+1}_j=\eta(\unpj)$, etc. By using the compactness
properties of $u_\Dt$, in particular that $u_\Dt(t,\dott)\in BV(B)$
and that $\px G(u_\Dt)$ is bounded, it follows by a standard
Lax--Wendroff calculations, see \cite[Thm.~3.4]{MR3443431} that
\begin{align*}
  \lim_{\Dt\to 0} \eqref{eq:tpart}&= \int_{\Omega_T}
                                    \eta(u)\pt\test\,dxdt - \int_B \eta(u(T,x)\test(T,x)\,dx + \int_B
                                    \eta(u_0(x))\test(0,x)\,dx,\\ 
  \intertext{and}
  \lim_{\Dt\to 0} \eqref{eq:convpart}&=
                                       -\int_{\Omega_T} \eta'(u)k\px G(u)\px\test \,dxdt.
\end{align*}
Regarding \eqref{eq:squarepart}, by Jensen's inequality
\begin{align*}
  \left(\Dp
  \big(\eta^\prime\big)^{n+1}_j\right)\left(\Dp G^{n+1}_j\right)\test^n_j&=
                                                                           \eta''(u^{n+1}_{j+1/2}) \Dp\unpj \Dp G^{n+1}_j\\
                                                                         &=\eta''(u^{n+1}_{j+1/2}) \left(\Dp \unpj\right)^2
                                                                           \Bigl[ \frac{1}{\Dp \unpj} \int_{\unpj}^{u^{n+1}_{j+1}}
                                                                           \left(g'(s)\right)^2 \,ds\Bigr]\\
                                                                         &\ge \eta''(u^{n+1}_{j+1/2}) \left(\Dp \unpj\right)^2 
                                                                           \Bigl[ \frac{1}{\Dp \unpj} \int_{\unpj}^{u^{n+1}_{j+1}}
                                                                           g'(s) \,ds\Bigr]^2\\
                                                                         &=\eta''(u^{n+1}_{j+1/2}) \left(\Dp g^{n+1}_j\right)^2,
\end{align*}
where $u^{n+1}_{j+1/2}$ is between $\unpj$ and $u^{n+1}_{j+1}$ and
$g(u)=\int^u \sqrt{G'(s)}\,ds$. Since $\eta''$ is continuous,
``$\eta''(u^{n+1}_{j+1/2}) \to \eta''(u)$'' as $\Dt\to 0$. We also
have that $g(u_\Dt)\to g(u)$ and that $\px g(u_\Dt)\to \px g(u)$
weakly. Therefore
\begin{align*}
  \lim_{\Dt\to 0} \Dt\Dx \sum_{n=0}^M \sum_{j=0}^N &\kjp\left(\ddp
                                                     \big(\eta^\prime\big)^{n+1}_j\right)\left(\ddp G^{n+1}_j\right)\test^n_j \\
                                                   &\ge
                                                     \lim_{\Dt\to 0} \Dt\Dx \sum_{n=0}^M \sum_{j=0}^N  \kjp
                                                     \eta''(u^{n+1}_{j+1/2}) \left(\ddp g^{n+1}_j\right)^2\test^n_j\\
                                                   &=\int_{\Omega_T} \eta''(u) k \left(\px g(u)\right)^2 \test\,dxdt.
\end{align*}
We have now shown that the limit $u$ is an entropy solution, i.e., we
have proved to following theorem:
\begin{theorem}
  \label{thm:scheme}
  Assume that $G$ is a non-decreasing, Lipschitz continuous function,
  that $k\in C^1([0,1])$ is a strictly positive function, and that the
  initial data $u_0\in BV([0,1])\cap L^\infty([0,1])$. Then the
  sequence $\seq{u_\Dt}_{\Dt>0}$ defined by \eqref{eq:diff_uDt},
  \eqref{eq:diff_initialcond} and the scheme \eqref{eq:implicit}
  converges to an entropy solution of \eqref{eq:boundval} for
  $B=(0,1)$.
\end{theorem}
\begin{proof}
  The compactness and convergence of a subsequence is already proved
  above, and we observe that since the entropy solution is unique, the
  whole sequence converges.
\end{proof}
\subsection{An example}\label{subsec:numex}
We present an example of how this method works in practice. Let
\begin{equation}
  \label{eq:Gnumex}
  G(u)=
  \begin{cases}
    0 &u<0, \\
    u(2-u) & 0\le u \le 1,\\
    1 & 1<u,
  \end{cases}
\end{equation}
and
\begin{equation}
  \label{eq:initial_numex}
  u_0(x)=2\sin(2\pi x).
\end{equation}
In Figure~\ref{fig:1} we show the approximate solution for $t=0$,
$t=0.07$, $t=0.13$ and $t=0.2$, computed with the implicit finite difference scheme
 \eqref{eq:implicit} with
$N=512$, $\Dx=1/513$, and $\Dt=0.01\Dx$. The nonlinear equation was
solved numerically using Newton iteration which terminated when the
error was less than $0.1\Dx^2$.
\begin{figure}[h]
  \centering
  \begin{tabular}{lr}
    \includegraphics[width=0.45\linewidth]{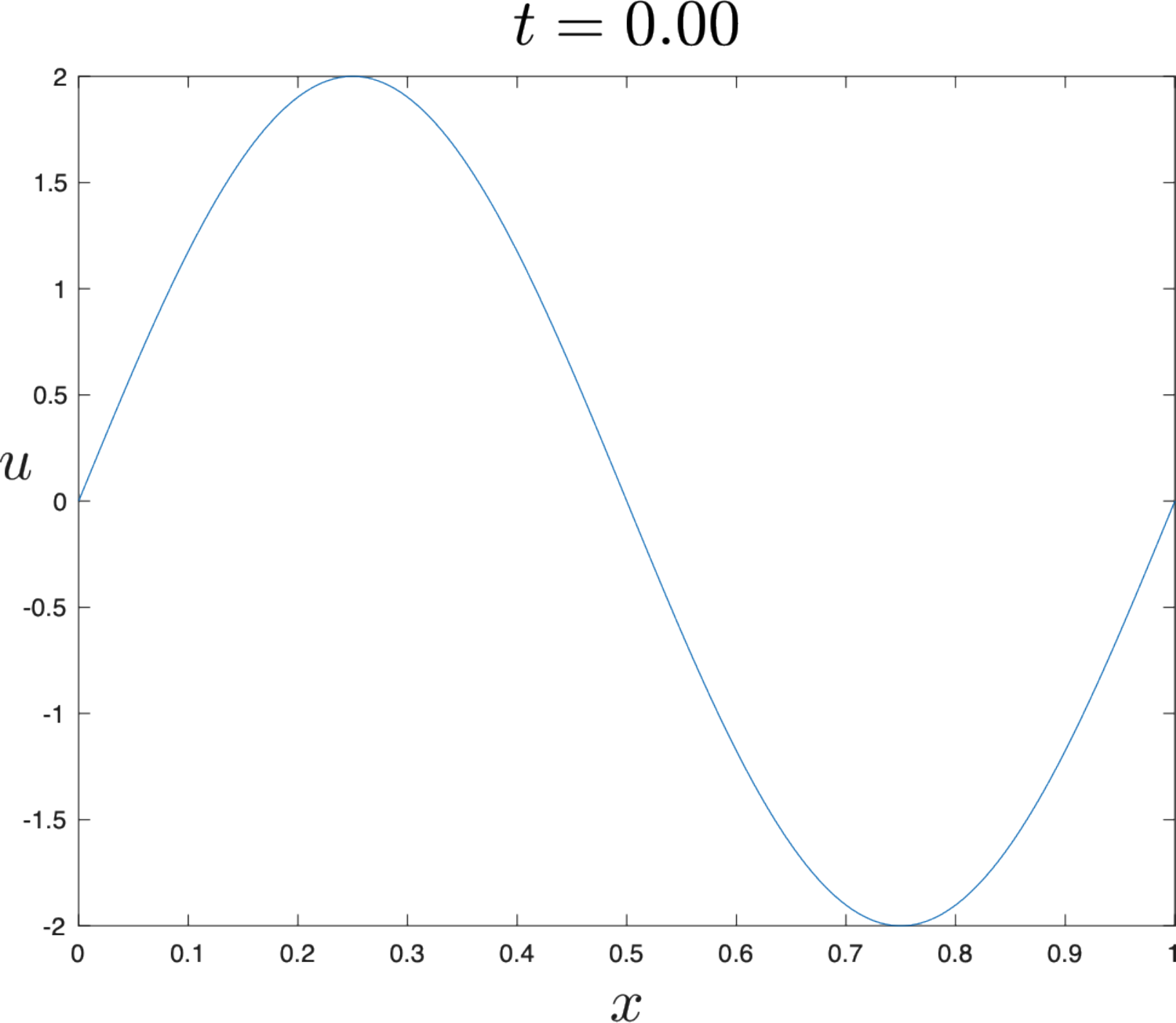}
    &\includegraphics[width=0.45\linewidth]{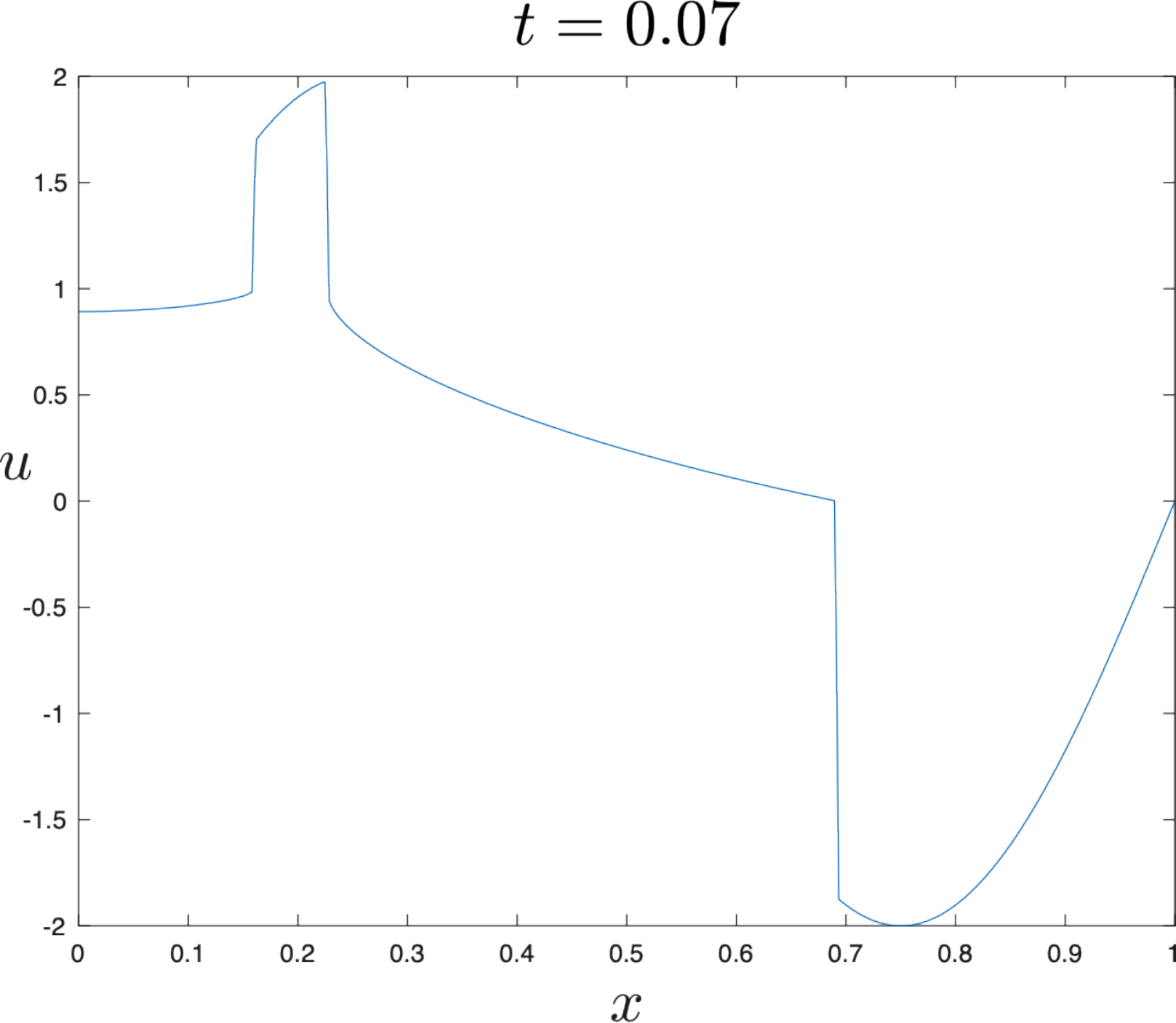}\\
    \includegraphics[width=0.45\linewidth]{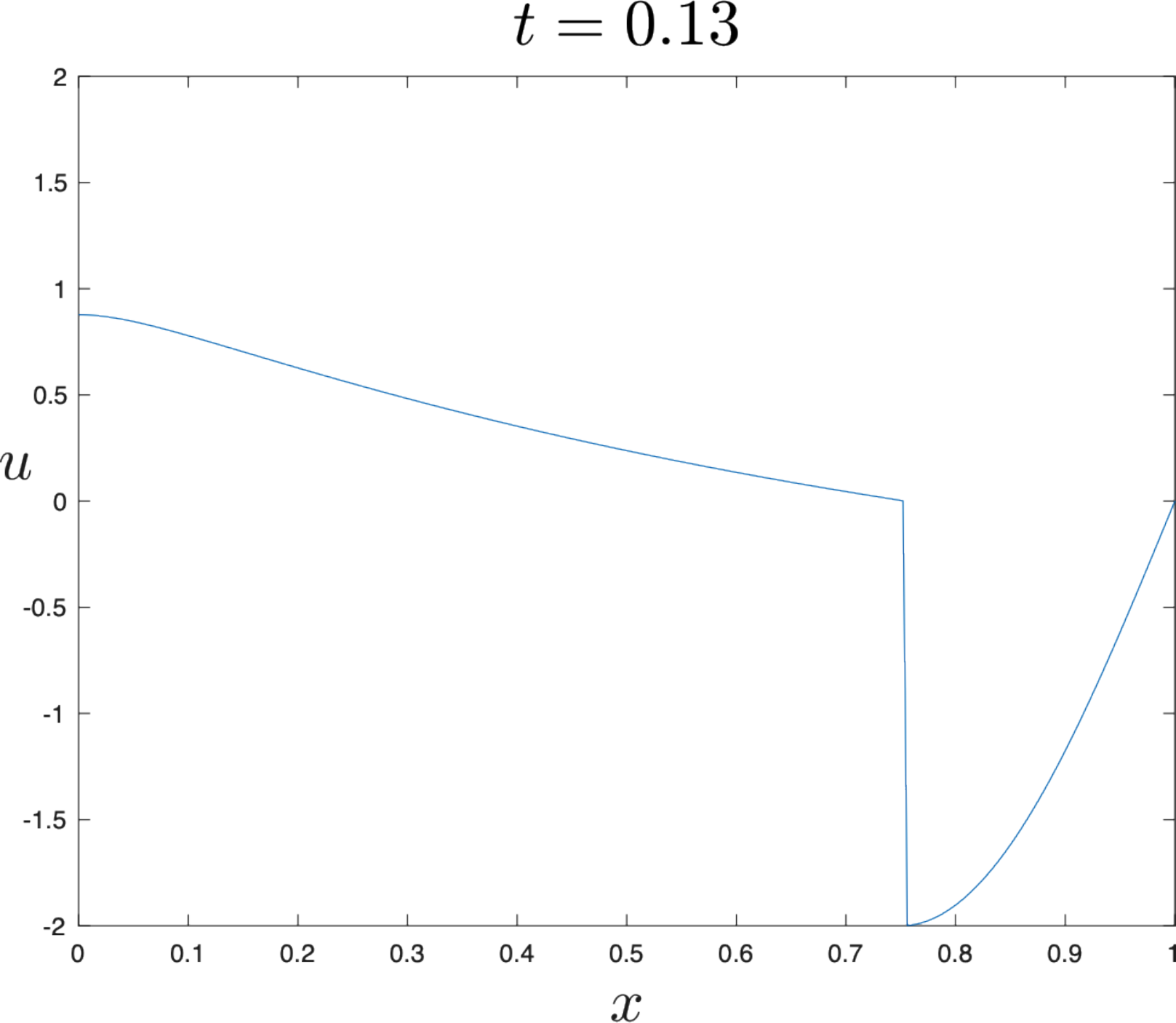}
    &\includegraphics[width=0.45\linewidth]{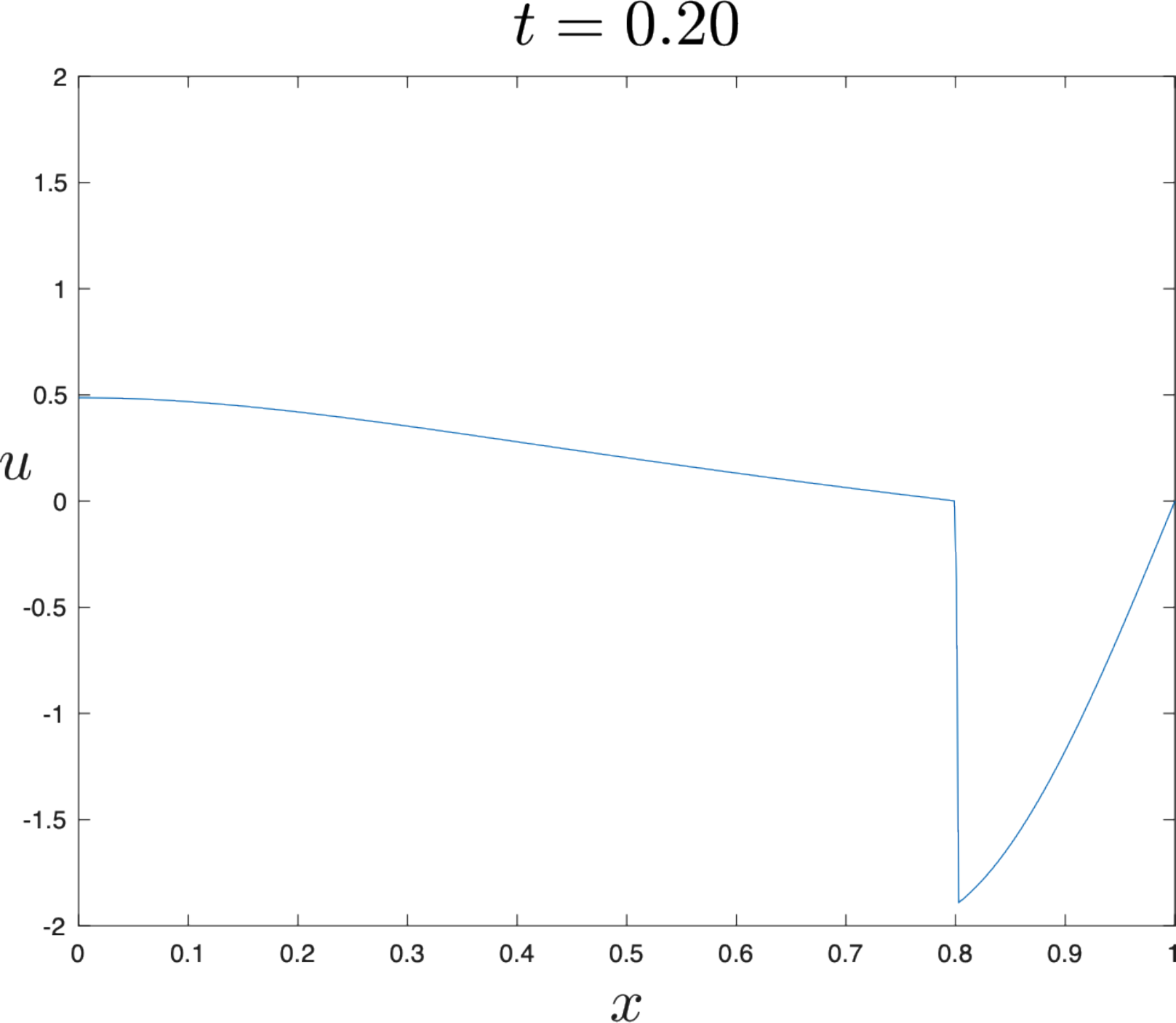}
  \end{tabular}
  \caption{The approximate solution for various times.}
  \label{fig:1}
\end{figure}
In Figure~\ref{fig:2} we also show the same approximate solution as a
function of $(t,x)$.
\begin{figure}[h]
  \centering \includegraphics[width=0.65\linewidth]{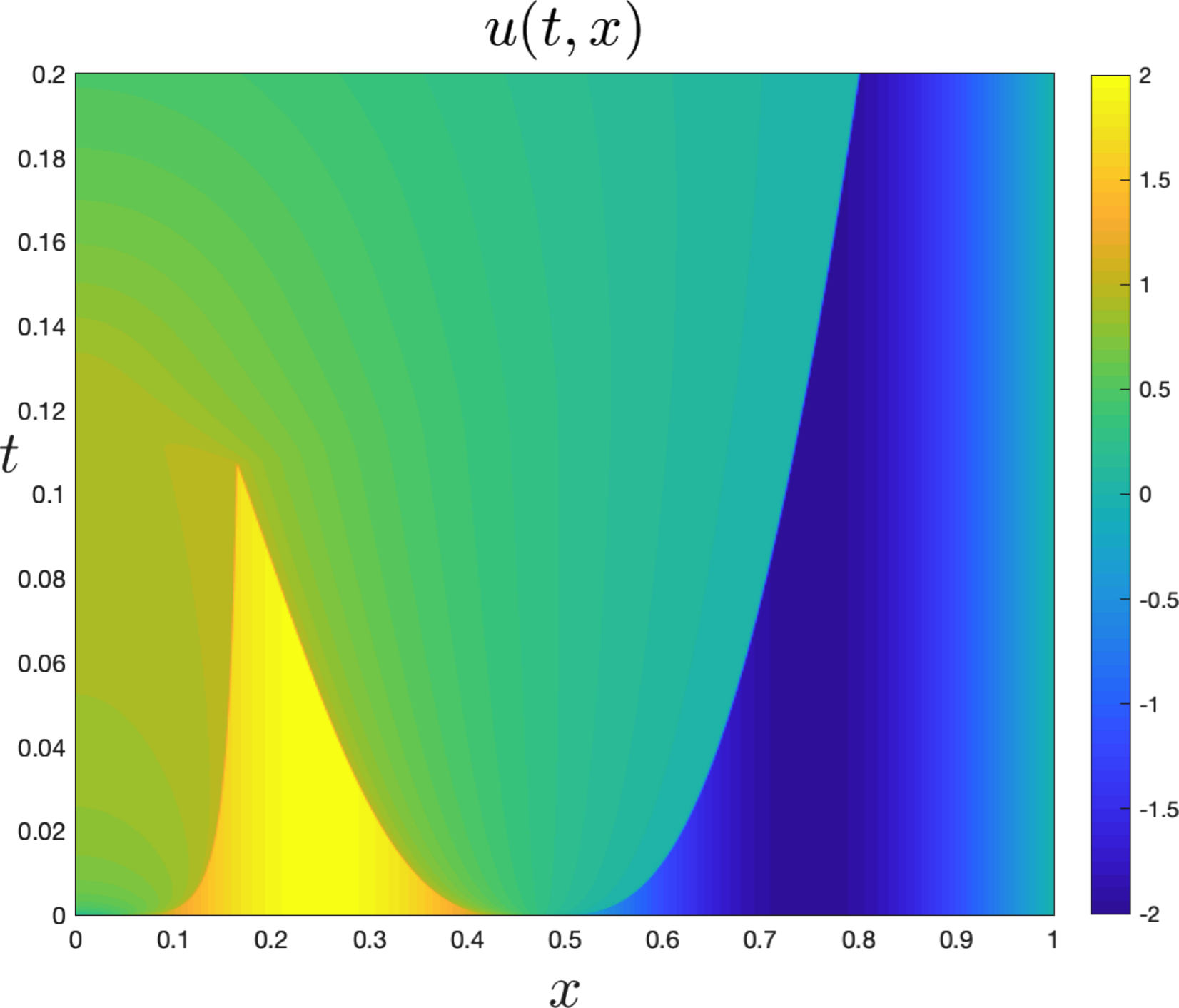}
  \caption{The appoximate solution for $0\le t\le 0.2$.}
  \label{fig:2}
\end{figure}


\begin{thebibliography}{10}

\bibitem{MR539218}
H.~Br\'{e}zis and M.~G. Crandall.
\newblock Uniqueness of solutions of the initial-value problem for
  {$u_{t}-\Delta \varphi (u)=0$}.
\newblock {\em J. Math. Pures Appl. (9)}, 58(2):153--163, 1979.

\bibitem{MR1769093}
R.~B\"{u}rger, S.~Evje, and K.~H. Karlsen.
\newblock On strongly degenerate convection-diffusion problems modeling
  sedimentation-consolidation processes.
\newblock {\em J. Math. Anal. Appl.}, 247(2):517--556, 2000.

\bibitem{MR1709116}
J.~Carrillo.
\newblock Entropy solutions for nonlinear degenerate problems.
\newblock {\em Arch. Ration. Mech. Anal.}, 147(4):269--361, 1999.

\bibitem{Chen2001}
G.-Q. Chen and H.~Frid.
\newblock On the theory of divergence-measure fields and its applications.
\newblock {\em Bol. Soc. Brasil. Mat. (N.S.)}, 32(3):401--433, Oct 2001.

\bibitem{MR2149515}
G.-Q. Chen and K.~H. Karlsen.
\newblock Quasilinear anisotropic degenerate parabolic equations with
  time-space dependent diffusion coefficients.
\newblock {\em Commun. Pure Appl. Anal.}, 4(2):241--266, 2005.

\bibitem{MR2187640}
G.-Q. Chen and K.~H. Karlsen.
\newblock {$L^1$}-framework for continuous dependence and error estimates for
  quasilinear anisotropic degenerate parabolic equations.
\newblock {\em Trans. Amer. Math. Soc.}, 358(3):937--963, 2006.

\bibitem{MR3136903}
P.~G. Ciarlet.
\newblock {\em Linear and {N}onlinear {F}unctional {A}nalysis with
  {A}pplications}.
\newblock Society for Industrial and Applied Mathematics, Philadelphia, PA,
  2013.

\bibitem{MR2734454}
H.~Hanche-Olsen and H.~Holden.
\newblock The {K}olmogorov--{R}iesz compactness theorem.
\newblock {\em Expo. Math.}, 28(4):385--394, 2010.

\bibitem{MR3964217}
H.~Hanche-Olsen, H.~Holden, and E.~Malinnikova.
\newblock An improvement of the {K}olmogorov--{R}iesz compactness theorem.
\newblock {\em Expo. Math.}, 37(1):84--91, 2019.

\bibitem{MR3443431}
H.~Holden and N.~H. Risebro.
\newblock {\em Front {T}racking for {H}yperbolic {C}onservation {L}aws}, volume
  152 of {\em Applied Mathematical Sciences}.
\newblock Springer, Heidelberg, second edition, 2015.

\bibitem{MR3721873}
H.~Holden and N.~H. Risebro.
\newblock The continuum limit of {F}ollow-the-{L}eader models---a short proof.
\newblock {\em Discrete Contin. Dyn. Syst.}, 38(2):715--722, 2018.

\bibitem{MR3845580}
H.~Holden and N.~H. Risebro.
\newblock Follow-the-leader models can be viewed as a numerical approximation
  to the {L}ighthill--{W}hitham--{R}ichards model for traffic flow.
\newblock {\em Netw. Heterog. Media}, 13(3):409--421, 2018.

\bibitem{MR4001759}
H.~Holden and N.~H. Risebro.
\newblock Models for dense multilane vehicular traffic.
\newblock {\em SIAM J. Math. Anal.}, 51(5):3694--3713, 2019.

\bibitem{MR1941794}
K.~H. Karlsen and M.~Ohlberger.
\newblock A note on the uniqueness of entropy solutions of nonlinear degenerate
  parabolic equations.
\newblock {\em J. Math. Anal. Appl.}, 275(1):439--458, 2002.

\bibitem{MR1974417}
K.~H. Karlsen and N.~H. Risebro.
\newblock On the uniqueness and stability of entropy solutions of nonlinear
  degenerate parabolic equations with rough coefficients.
\newblock {\em Discrete Contin. Dyn. Syst.}, 9(5):1081--1104, 2003.

\bibitem{MR3246807}
K.~H. Karlsen, N.~H. Risebro, and E.~B. Storr{\o}sten.
\newblock {$L^1$} error estimates for difference approximations of degenerate
  convection-diffusion equations.
\newblock {\em Math. Comp.}, 83(290):2717--2762, 2014.

\bibitem{MPS}
A.~Maugeri, D.~Palagachev, and L.~Softova.
\newblock {\em Elliptic and Parabolic Equations with Discontinuous
  Coefficients, Volume 109}.
\newblock Wiley-VCH, Berlin-Weinheim-New
  York-Chichester-Brisbane-Singapore-Toronto, 2000.

\bibitem{MR0264232}
A.~I. Vol'pert and S.~I. Hudjaev.
\newblock The {C}auchy problem for second order quasilinear degenerate
  parabolic equations.
\newblock {\em Mat. Sb. (N.S.)}, 78 (120):374--396, 1969.

\end{thebibliography}
\end{document}